\documentclass{siamltex}
\usepackage{graphicx}
\usepackage{amsmath}
\usepackage{latexsym}
\usepackage{amssymb}
\usepackage{color}
\usepackage{amsfonts}
\usepackage{ulem}
\usepackage{algorithm,algorithmic}
\usepackage{fancyhdr}
\usepackage{multirow}
\usepackage{epsfig}

\newtheorem{remark}{Remark}[section]

\makeatletter
\newcommand{\ssymbol}[1]{^{\@fnsymbol{#1}}}
\makeatother

\usepackage[mathscr]{euscript}

\title{On some tensor tubal-Krylov subspace methods using  the T-product}
\author{A. El Ichi\thanks{Laboratoire de Mathématiques, Informatique et Applications, S\'ecurit\'e de l'Information LABMIA-SI, University Mohamed V, Rabat Morocco} \and K. Jbilou \thanks{LMPA, 50 rue F. Buisson, ULCO Calais, France; Mohammed VI Polytechnic University, Green City, Morocco; jbilou@univ-littoral.fr }  \and R. Sadka\footnotemark[1]}

\date{}

\begin{document}

	\maketitle

\maketitle

\begin{abstract}

\noindent   In this paper, we will introduce some new tubal-Krylov subspace methods for solving some linear tensor equations. Using the well known tensor T-product, we will in particular define the tensor tubal-global GMRES that could be seen as a generalization of the  global GMRES. We also give a new tubal-version of the tensor Golub-Kahan algorithm. To this end, we first introduce some new tensor-tensor products and also some new definitions. The presented numerical tests compare the two methods and show the efficiency of the proposed procedures..\\


\end{abstract}

\section{Introduction}
We consider the following tensor linear equation 
	\begin{equation}\label{eq1}
	\mathscr{A} \star\mathscr{X} = \mathscr{B},
	\end{equation}
	where $\mathscr{A}  $, $\mathscr{X}$,  and $\mathscr{B}$ are  three-way tensors  and $\star$   is the T-product  introduced   in \cite{klimer3,klimer2}.\\ A tensor is  a multidimensional array of data. The number of indices of a tensor is called modes or ways.  	Notice that a scalar can be regarded as a zero mode tensor, first mode tensors are vectors and matrices are second mode tensor. The order of a tensor is the dimensionality of the array needed to represent it, also known as
ways or modes. 
For a given 3-mode (or order-3) tensor $ \mathscr {X}\in \mathbb{R}^{n_{1}\times n_2 \times n_{3}}$, the notation $x_{i_{1},{i_2},i_{3}}$ for the element $\left(i_{1},i_2,i_{3} \right) $ of the tensor $\mathscr {X}$. \\
Fibers are the higher-order analogue of matrix rows and columns. A fiber is
defined by fixing all the indexes  except  one. A matrix column is a mode-1 fiber and a matrix row is a mode-2 fiber. Third-order tensors have column, row and tube fibers/scalars. An element ${\rm c }\in \mathbb{R}^{1\times 1 \times n}$ is called a  fiber/scalar tube of length $n$. More details are found in \cite{Jbiloubeik1,kolda,klimer3} .  \\
Tensors have been widely used and applied in different areas and especially in color image, video restoration or compression \cite{elguide1,elguide2,klimer2,klimer3}. Other applications of tensors  in  modern sciences, e.g.,  signal processing \cite{lb},  data mining \cite{lxnm}, tensor complementarity problems, computer vision,  see  \cite{ smile} for  more details. Recent tensor approaches were used  for numerically solving PDEs in \cite{dwwm}.\\
In the present work, we develop  the    tensor tubal  global GMRES method (TGGMRES)   for solving  tensor   system of  equations (\ref{eq1}) that could be considered as a generalisation of the global GMRES developed in \cite{Jbilou} and also a tensor tubal Golub Kahan method.  To this end, we introduce some new tensor products with some new related algebraic properties.

The paper is organized as follows: In Section 2, we give notations and definitions related to the T-product. In Section 3, we develop  some new tensor products and give some algebraic properties. After defining a tubal-global QR factorisation algorithm  we propose  in Section 4,  the tensor tubal-global Arnoldi process that allows us to introduce the tubal-global GMRES method. Section 5 is devoted to the tensor tubal Golub Kahan method. Finally, some numerical tests are reported in Section 6.




\section{Definitions and notations}\label{sec:section2}

In this section we recall some definitions and properties of  the T-product which is based on the Discrete Fourier Transformation (DFT)  defined  on 
a vector $v \in {\mathbb{C}}^n$ as follows 
	\begin{equation}
	\label{dft1}
	\tilde v= F_nv \in {\mathbb{C}}^n,
	\end{equation}
	where $F_n$ is the Fourier complex $n \times n$   matrix whose components are  given by  
	\begin{equation}
	\label{dft2}
	(F_n)_{ij}=  \omega^{(i-1)(j-1)},\; i,j=1,\ldots,n,
	\end{equation}
where $\omega=e^{\frac{-2 \pi i}{n}}$ with $i^2=-1$.
The cost of computing the vector $\tilde v$ directly from \eqref{dft1} is $O(n^2)$. Using the Fast Fourier Transform, it will cost only  $O(nlog(n))$ and this makes the FFT very fast for large problems.

In this part, we briefly review some concepts and notations related to the T-Product, see \cite{braman,klimer3,klimer2} for more details. 	Let $\mathscr {A} \in \mathbb{R}^{n_{1}\times n_{2}\times n_{3}} $ be a third-order tensor, then the operations ${\rm bcirc}$,   unfold and fold are defined by
	$${\rm bcirc}(\mathscr {A})=\left( {\begin{array}{*{20}{c}}
		{{A_1}}&{{A_{{n_3}}}}&{{A_{{n_{3 - 1}}}}}& \ldots &{{A_2}}\\
		{{A_2}}&{{A_1}}&{{A_{{n_3}}}}& \ldots &{{A_3}}\\
		\vdots & \ddots & \ddots & \ddots & \vdots \\
		{{A_{{n_3}}}}&{{A_{{n_{3 - 1}}}}}& \ddots &{{A_2}}&{{A_1}}
		\end{array}} \right)   \in {\mathbb{R}}^{ n_1n_3 \times n_2n_3},$$
$${\rm unfold}(\mathscr {A} ) = \begin{pmatrix}
A_{1},
A_{2},
\ldots
A_{n_{3}}\end{pmatrix}^T \in \mathbb{R}^{n_{1}n_{3}\times n_{2}},  \qquad {\rm fold}({\rm unfold}(\mathscr {A}) ) =  \mathscr {A}.$$ 
Let $\widetilde {\mathscr {A}}$ be the tensor obtained by applying the DFT on all the 3-mode tubes of the tensor $\mathscr {A}$. With the Matlab command ${\tt fft}$, we have
\begin{equation*}
\widetilde {\mathscr {A}}= {\tt fft}(\mathscr {A},[\, ],3)\; \; {\rm and }\;\; \mathscr {A}={\tt ifft} (\widetilde {\mathscr {A}}, [ \,],3),
\end{equation*}
where ${\tt ifft}$ denotes the Inverse Fast Fourier Transform.  The tensor $\widetilde {\mathscr {A}}$ can  also be obtained   using  the  3-mode product \cite{klimer2}, as  follows 
\begin{equation}\label{fftifft}
\widetilde {\mathscr {A}}=   {\mathscr {A}} \times_3 F_{n_3}\;\; {\rm and}\; \;{\mathscr {A}}= \widetilde {\mathscr {A}}\times_3 F_{n_3}^{-1}
\end{equation}   where  $\times_3$ is the  3-mode product defined in \cite{kolda}.  
\noindent 	Let ${\bf A}$ be the block diagonal matrix 
	\begin{equation}\label{dft9}
	{\bf A}= {\rm BlockDiag}(\widetilde {\mathscr {A}})= \left (
	\begin{array}{cccc}
	{A}^{(1)}& &&\\
	& {A}^{(2)}&&\\
	&&\ddots&\\
	&&&{A}^{(n_3)}\\
	\end{array}
	\right),
	\end{equation}
where  the matrices ${A}^{(i)}$'s are the frontal slices of the tensor $\widetilde {\mathscr {A}}$.
The block circulant matrix ${\rm bcirc}(\mathscr {A})$ can  be block diagonalized by using the DFT and this gives
\begin{equation}\label{dft8}
(F_{n_3} \otimes I_{n_1})\, {\rm bcirc}(\mathscr {A})\, 	(F_{n_3}^{*} \otimes I_{n_2})={\bf A},
\end{equation}	 	
As noticed in \cite{klimer2,lu}, the diagonal blocks of the matrix ${\bf A}$ satisfy the following property 
	\begin{equation}
	\label{f1}
	\left \{
	\begin{array}{ll}
	{A}^{(1)} \in {\mathbb{R}}^{n_1 \times n_2}\\
	conj({A}^{(i)})= A^{(n_3-i+2)},\;  i=2,\ldots,\lfloor \displaystyle \frac{{n_3}+1}{2} \rfloor\\
	\end{array}
	\right.
	\end{equation}
where 	$conj ({A}^{(i)})$ is the complex conjugate of the matrix ${A}^{(i)}$. Next we recall the definition of the T-product; see \cite{klimer2}.	

\medskip

\begin{definition}
	The T-product denoted by $\star $ between  two tensors
	$\mathscr {A} \in \mathbb{R}^{n_{1}\times n_{2}\times n_{3}} $ and $\mathscr {B} \in \mathbb{R}^{n_{2}\times m\times n_{3}} $ is the  ${n_{1}\times m\times n_{3}}$ tensor  given by:	
	$$\mathscr {A} \star \mathscr {B}={\rm fold}({\rm bcirc}(\mathscr {A}){\rm unfold}(\mathscr {B}) ).$$
\end{definition}
\noindent Notice that from the relation \eqref{dft9}, we can show that the   product  $\mathscr {C}=\mathscr {A} \star \mathscr {B}$ is equivalent to ${\bf C}= {\bf A}  {\bf B}$ where ${\bf A}= {\rm BlockDiag}(\widetilde {\mathscr {A}})$ given in \eqref{dft9} and ${\bf B}= {\rm BlockDiag}(\widetilde {\mathscr {B}})$. So, the efficient way to compute the T-product is to use Fast Fourier Transform (FFT). \\
Using the relation \eqref{f1}, the following algorithm allows us to compute in an efficient way the T-product of the tensors $\mathscr {A}$ and $\mathscr {B}$ .
\begin{algorithm}[H]
	\begin{algorithmic}[1]
		\REQUIRE 	Input  $\mathscr {A} \in \mathbb{R}^{n_{1}\times n_{2}\times n_{3}} $ and $\mathscr {B} \in \mathbb{R}^{n_{2}\times m\times n_{3}}$.
			\ENSURE: $ \mathscr {C}= \mathscr {A} \star \mathscr {B}  \in \mathbb{R}^{n_{1}\times m \times n_{3}}.$	
		\STATE Compute $\mathscr {\widetilde A}={\tt fft}(\mathscr {A},[ ],3)$ and $\mathscr {\widetilde B}={\tt fft}(\mathscr {B},[ ],3)$;%
		\FOR{ $i=1,\dots,n_3$}
		\STATE	 Compute each frontal slices of $\mathscr {\widetilde C}$
	 	$$C^{(i)}= \left \{
	 	\begin{array}{ll}
		 	A^{(i)} B^{(i)} , \quad \quad\quad  i=1,\ldots,\lfloor \displaystyle \frac{{n_3}+1}{2} \rfloor\\
		 	conj({C}^{(n_3-i+2)}),\quad \quad i=\lfloor \displaystyle \frac{{n_3}+1}{2} \rfloor+1,\ldots,n_3,
		 	\end{array}
	 		\right.$$
		\ENDFOR		
		\STATE Compute $\mathscr {C}={\tt ifft}(\widetilde {C},[\,],3)$.
	\end{algorithmic} 
	\caption{Computing the  T-product via FFT}
	\label{algo}
\end{algorithm}

\noindent For the T-product, we recall the following definitions; see \cite{klimer2} for more details.

\begin{definition}.\label{identens}
	\begin{enumerate}
		\item The identity tensor $\mathscr{I}_{n_{1}n_{1}n_{3}} $ is the tensor whose first frontal slice is the identity matrix $I_{n_1n_1}$ and the other frontal slices are all zeros.
		\item 	An $n_{1}\times n_{1} \times n_{3}$ tensor $\mathscr{A}$ is invertible, if there exists a tensor $\mathscr{X}$ of order  $n_{1}\times n_{1} \times n_{3}$  such that
		$$\mathscr{A}  \star \mathscr{X}=\mathscr{I}_{ n_{1}  n_{1}  n_{3}} \qquad \text{and}\qquad \mathscr{X}  \star \mathscr{A}=\mathscr{I}_{ n_{1}  n_{1}  n_{3}}.$$
		In that case, we set $\mathscr{X}=\mathscr{A}^{-1}$. 	It is clear that 	$\mathscr{A}$ is invertible if and only if   ${\rm bcirc}(\mathscr{A})$ is invertible.
		\item The transpose of $\mathscr{A}$  is obtained by transposing each of the frontal slices and then reversing the order of transposed frontal slices 2 through $n_3$. 
		\item If $\mathscr {A}$, $\mathscr {B}$ and $\mathscr {C}$ are tensors of appropriate orders, then
		$$(\mathscr {A} \star \mathscr {B}) \star \mathscr {C}= \mathscr {A} \star (\mathscr {B} \star \mathscr {C}).$$
		\item Suppose that $\mathscr {A}$ and $\mathscr {B}$ are two tensors such $\mathscr {A} \star \mathscr {B}$ and $ \mathscr {B}^T \star \mathscr {A}^T$ are  defined. Then  $$(\mathscr {A} \star \mathscr {B})^T= \mathscr {B}^T \star \mathscr {A}^T.$$
	\end{enumerate}
\end{definition}

\begin{definition} Let 	$\mathscr{A}$ and $\mathscr{B}$ two tensors in $\mathbb{R}^{n_1 \times n_2 \times n_3}$. Then
	\begin{enumerate}
		\item The scalar inner product is defined by
		$$\langle \mathscr{A}, \mathscr{B} \rangle = \displaystyle \sum_{i_1=1}^{n_1} \sum_{i_2=1}^{n_2}  \sum_{i_3=1}^{n_3} a_{i_1 i_2 i_3}b_{i_1 i_2 i_3}.$$
		\item The associated Frobenius norm  is defined by
		$$ \Vert \mathscr{A} \Vert_F=\displaystyle \sqrt{\langle  \mathscr{A} ,  \mathscr{A}  \rangle}.$$
	\end{enumerate}
\end{definition}
\medskip
\begin{remark}
	Another interesting way for computing the scalar product $\langle \mathscr{A}, \mathscr{B} \rangle $ and the associated Frobenius norm is as follows:
	$$\langle \mathscr{A}, \mathscr{B} \rangle = \displaystyle \frac{1}{n_3}  \langle  {\bf A}, {\bf B} \rangle\; {\rm and} \; \Vert \mathscr{A} \Vert_F= \displaystyle \frac{1}{\sqrt{n_3}} \Vert {\bf A} \Vert_F,$$
	where the block diagonal matrix $\bf A$  is defined by \eqref{dft9}.
\end{remark}

\begin{definition}
	\begin{enumerate}
		\item An $n_{1}\times n_{1} \times n_{3}$ tensor  $\mathscr{Q}$  is orthogonal if
		$$\mathscr{Q}^{T}   \star  \mathscr{Q}=\mathscr{Q} \star \mathscr{Q}^{T}=\mathscr{I}_{ n_{1}  n_{1}  n_{3}}.$$		
		\item 	A tensor is called f-diagonal if its frontal slices are orthogonal matrices. It is called upper triangular if all its frontal slices are upper triangular. 
		
	\end{enumerate}
\end{definition}

\begin{definition}\label{bloctens0}\cite{miaoTfunction} 
	Let  $\mathscr{A}   \in {\mathbb R}^{n_{1}\times m_{1} \times n_{3}} $, $\mathscr{B}\in {\mathbb R}^{n_{1}\times m_{2} \times n_{3}}$, $\mathscr{C}   \in {\mathbb R}^{n_{2}\times m_{1} \times n_{3}} $ and   $\mathscr{D}\in {\mathbb R}^{n_{2}\times m_{2} \times n_{3}}$ be  tensors. The block tensor
		$$\left[ {\begin{array}{*{20}{c}}
			{\mathscr{A}}&{\mathscr{B}} \\
			{\mathscr{C}}&{\mathscr{D}} \\
			\end{array}} \right]\in {\mathbb R}^{(n_{1}+n_2)\times (m_1+m_{2}) \times n_{3}} $$
		is defined by compositing the frontal slices of the four tensors.
	\end{definition}
	\begin{proposition} \label{propoblock}
		Let $\mathscr{A},\mathscr{A}_1   \in {\mathbb R}^{n \times s \times n_{3}} $, $\mathscr{B},\mathscr{B}_1\in {\mathbb R}^{n \times p \times n_{3}}$, $\mathscr{A}_2   \in {\mathbb R}^{\ell\times s \times n_{3}} $,  $\mathscr{B}_2\in {\mathbb R}^{\ell\times p \times n_{3}}$, $\mathscr{C}    \in {\mathbb R}^{s \times n \times n_{3}}$, $\mathscr{D}    \in {\mathbb R}^{p \times n \times n_{3}}$  and  $\mathscr{F}    \in {\mathbb R}^{n \times n \times n_{3}} $ . Then 
		\begin{enumerate}
			\item $\mathscr {F} \star	\left[\mathscr {A} \;\mathscr {B}\right]=	\left[ \mathscr {F}\star \mathscr {A} \;\;\mathscr {F}\star\mathscr {B}\right]\in \mathbb{R}^{ n\times (s+p)\times n_3  }$
			\item $\begin{bmatrix}
			\mathscr {C} \\
			\mathscr {D} 
			\end{bmatrix}\star \mathscr {F}=\begin{bmatrix}
			\mathscr {C}\star \mathscr {F} \\
			\mathscr {D} \star \mathscr {F}
			\end{bmatrix}\in \mathbb{R}^{ (s+p)    \times n\times n_3}$
			\item $\left[ \mathscr {A} \;\mathscr {B}\right]\star\begin{bmatrix}
			\mathscr {C} \\
			\mathscr {D} 
			\end{bmatrix}=\mathscr {A}\star\mathscr {C}+\mathscr {B}\star \mathscr {D}\in \mathbb{R}^{   n\times n \times n_3}$\\
			\item $\begin{bmatrix}
			\mathcal {A}_1&\mathcal {B}_1 \\
			\mathcal {A}_2&\mathcal {B}_2
			\end{bmatrix}\star  \begin{bmatrix}
			\mathscr {C} \\
			\mathscr {D} 
			\end{bmatrix}=\begin{bmatrix}
			\mathscr {A}_1\star\mathscr {C}+\mathscr {B}_1\star\mathscr {D} \\
			\mathscr {A}_2\star\mathscr {C}+\mathscr {B}_2\star\mathscr {D}
			\end{bmatrix}\in \mathbb{R}^{  (\ell+n)   \times n\times n_3}$
		\end{enumerate}
	\end{proposition}

\medskip
\noindent Next, we  introduce now  the \textit{T-trace } transformation.
\begin{definition}
	Let $\mathscr{A}$  be a  tensor in  $ \mathbb{R}^{n_{1}\times n_1\times n_{3}}$. The tensor    $\text{T-trace}$ of $\mathscr{A}$ is a fiber-tensor of $\mathbb{R}^{1\times 1\times n_{3}}$ defined such that  its  i-th  frontal slice   is the trace of i-th frontal slice of $\mathscr {\widetilde A}$, for $i=1,\ldots,n_3$.  
\end{definition}
\noindent 	The $\text{T-trace}(\mathscr{A})$ can be computed by  the   following Algorithm .

\begin{algorithm}[H]
	\begin{algorithmic}[1]
		\REQUIRE 	Input  $\mathscr{A}\in {\mathbb R}^{n_{1}\times n_{1} \times n_{3}} $.
			\ENSURE:  ${\rm \bf z} \in {\mathbb R}^{1\times 1 \times n_{3}} $
		\STATE Set $\mathscr {\widetilde A}=\text{{\tt fft}}(\mathscr{ A},[\,],3) $,%
		\FOR{ $i=1,\dots,n_3$}
		\STATE	$	 { {\rm \bf z}}^{(i)} =     trace( {A}^{(i)}) $, \quad (trace matrix)
		\ENDFOR
		\STATE $  {\rm \bf z} =\text{{\tt ifft}} ( {\widetilde{{\rm \bf z}}},[\,],3)$.		 
	\end{algorithmic} 
	\caption{Tensor T-trace}
	\label{T-0trace3}
\end{algorithm}

\section{New tensor products}

\noindent In this section, we introduce  some  new tensor products,   that will   be used for simplifying the algebraic computations of the main results.

\begin{definition}\label{dfttubefibr01}  
	Let  ${\rm \bf a}  \in {\mathbb R}^{1\times 1 \times n_{3}} $ and $\mathscr{B}= \mathscr{B}(i,j,:)={\rm \bf b}_{ij}  \in {\mathbb R}^{m_{1}\times m_{2} \times n_{3}} $ for $i=1,\ldots,m_{1},\;j =1,\ldots,m_{2}$. Then,  the   product $({\rm \bf a}\divideontimes\mathscr{B})$ is an $(m_{1}\times m_{2} \times n_{3})$ tensor defined by 
	\begin{eqnarray*}
			{\rm \bf a}\divideontimes\mathscr{B} =\begin{pmatrix}
				{\rm \bf a}\star {\rm \bf b}_{11}  &\ldots&{\rm \bf a}\star {\rm \bf b}_{1m_{2}}  \\
				\vdots&\ddots&\vdots \\
				{\rm \bf a}\star {\rm \bf b}_{m_{1}1}  &\ldots&{\rm \bf a}\star {\rm \bf b}_{m_{1}m_{2}}  \\
			\end{pmatrix}
	\end{eqnarray*}
\end{definition}
\begin{remark}  The $\divideontimes$ product is a generalisation of the product of a scalar with a matrix where   the tubal-fiber   plays the role of a scalar
\end{remark}

\subsection{The T-Kronecker and the Tubal-inner  products }

\noindent In the following we introduce the T-Kronecker product between two tensors as a generalisation of the classical Kronecker product for matrices. 

\noindent
\begin{definition}
	Let	$\mathscr{A}   \in {\mathbb R}^{n_{1}\times n_{2} \times n_{3}} $ and   $\mathscr{B}\in {\mathbb R}^{m_{1}\times m_{2} \times n_{3}}$. The  T-Kronecker product  $\mathscr{A}\circledast \mathscr{B}$  between $ \mathscr{A}$ and $\mathscr{B}$ is  the $n_{1}m_{1}\times n_{2}m_{2} \times n_{3} $ tensor  given by	:
	\begin{equation*}
	(\mathscr{A}\circledast\mathscr{B})   =       \widetilde{(\mathscr{A}\circledast\mathscr{B})}  \times_3 F_{n_3}^{-1}     
	\end{equation*}
	\noindent where the  $i$-th  frontal  slice ( $i=1,\ldots,n_3 $) of $  \widetilde{(\mathscr{A}\circledast\mathscr{B})} $ is  given  by, 
	\begin{equation*}
		(\mathscr{A}\circledast\mathscr{B})^{(i)} =
		({\mathscr {A}} \times_3 F_{n_3})^{(i)}\otimes (\mathscr{B}\times_3 F_{n_3})^{(i)} 
		\end{equation*}
	\noindent    where $\otimes$ is the Kronecker product  between two  matrices. The matrices  $     ({\mathscr {A}} \times_3 F_{n_3})^{(i)}$ and $  ({\mathscr {B}} \times_3 F_{n_3})^{(i)}$  are  the i-th frontal  slices  of $\widetilde {\mathscr {A}}$ and $\widetilde {\mathscr {B}}$,   respectively.
\end{definition}	

\noindent The  T-Kronecker product of two tensors can be computed  by the following algorithm.

\begin{algorithm}[H]
	\begin{algorithmic}[1]
		\REQUIRE 	 $\mathscr{A}\in {\mathbb R}^{n_{1}\times n_{2} \times n_{3}} $ and   $\mathscr{B}\in {\mathbb R}^{m_{1}\times m_{2} \times n_{3}}$  .
		\ENSURE: $(\mathscr{A}\circledast\mathscr{B})$ is  the tensor of  size  $n_{1}m_{1}\times n_{2}m_{2}$.
		\STATE Set $\mathscr {\widetilde A}=\text{{\tt fft}}(\mathscr{ A},[\,],3) $   and   $ \mathscr {\widetilde B}=\text{{\tt fft}}(\mathscr{B},[\,],3) $,%
		\FOR{ $i=1,\dots,n_3$}
		\STATE $ (\mathscr{A}\circledast\mathscr{B})^{ (i)} =     ( {A}^{(i)}\otimes   {B}^{(i)})$. 
		\ENDFOR
		\STATE $ (\mathscr{A}\circledast\mathscr{B}) =\text{{\tt ifft}} (\widetilde{(\mathscr{A}\circledast\mathscr{B})},[\,],3)$.
		 
	\end{algorithmic} 
	\caption{Tensor T-Kronecker}
	\label{krontensot}
\end{algorithm}

\noindent\begin{proposition}\label{kronprop1}
	Let $\mathscr{A}\in {\mathbb R}^{n_{1}\times n_{2} \times n_{3}}$, $\mathscr{B}\in {\mathbb R}^ {m_{1}\times m_{2} \times n_{3}}$, $\mathscr{C}\in {\mathbb R}^{n_{2}\times r_{1} \times n_{3}}$ and $\mathscr{D}\in {\mathbb R}^{m_{2}\times r_{2} \times n_{3}}$. Then  we have the following properties
	
	\begin{enumerate}
		\item
		$(\mathscr{A}\circledast\mathscr{B})^{T}=\mathscr{A}^{T}\circledast  \mathscr{B}^{T}$
		\item	  $(\mathscr{A}\circledast\mathscr{B}) \star (\mathscr{C}\circledast\mathscr{D}) =(\mathscr{A}\star\mathscr{C})\circledast (\mathscr{B}\star\mathscr{D})$
		\item
		If $\mathscr{A}\in {\mathbb R}^{n\times n \times z}$ and $\mathscr{B}\in {\mathbb R}^{p\times p \times z}$ are invertible then $(\mathscr{A}\circledast\mathscr{B})^{-1}$ is invertible and we have : 
		\[  (\mathscr{A}\circledast\mathscr{B})^{-1} =\mathscr{A}^{-1}\circledast \mathscr{B}^{-1}\]
	\end{enumerate}
\end{proposition}
\medskip
\noindent
\begin{proof}
	Obviously, the  results stems directly from the properties of the matrix-Kronecker.  In fact, 
	for $i=1,\ldots,n_3$, we  have  : \begin{align*}
	((\mathscr{A}^{T}\circledast  \mathscr{B}^{T})\times_3 F_{n_3})^{(i)}&= ( (\mathscr{A}\times_3 F_{n_3})^{(i)T} \otimes(\mathscr{B}\times_3 F_{n_3})^{(i)T})\\
	&= ((\mathscr{A}\times_3 F_{n_3})^{(i)} \otimes(\mathscr{B}\times_3 F_{n_3})^{(i)})^T\\
	&=  ((\mathscr{A} \circledast  \mathscr{B} ) ^{T}\times_3 F_{n_3})^{(i)}	\end{align*}
	which shows that 	$(\mathscr{A}\circledast\mathscr{B})^{T} =(\mathscr{A}^{T}\circledast  \mathscr{B}^{T})$.  
	The two other properties are shown in the same way.
\end{proof}	

\noindent Next, we define a  new Tubal-inner product that will be used later.
\noindent\begin{definition}{(Tubal-inner product)}
	For  $\mathscr{X},\mathscr{Y}$ two tensors in $\mathbb{R}^{n_{1}\times s\times n_{3}}$, we define the Tubal-inner  products  $\left\langle {.,.} \right\rangle_T$
	is   
	defined by:
	\begin{align}\label{bilinearform3d}
	\begin{cases}
	\mathbb{R}^{n_{1}\times s\times n_{3}}\times \mathbb{R}^{n_{1}\times s\times n_{3}}   & \longrightarrow  \mathbb{R}^{1\times 1\times n_{3}} \\
	\qquad  \qquad (\mathscr{X},    \mathscr{Y} ) \qquad   &\longrightarrow  \langle \mathscr{X}, \mathscr{Y} \rangle_T =  \text{T-trace}(\mathscr{X}^{T}\star \mathscr{Y}).
	\end{cases}.
	\end{align}
\end{definition}
\noindent Let  $\mathscr{X}_{1},\ldots, \mathscr{X}_{\ell}$ be a collection of $\ell$ third tensors in
$\mathbb{R}^{n_{1}\times s\times n_{3}}$, if
\begin{align*}
\left\langle {\mathscr{X}_{i}, \mathscr{X}_{j}} \right\rangle_T =
\begin{cases}
\alpha_{i}{\rm \bf e}&i= j \\
0&i\neq j,
\end{cases}.
\end{align*}
where  $\alpha_{i}$ is a non-zero scalar and  ${\rm \bf e}$ is the tubal-fiber such that ${\rm unfold}({\rm \bf e})  =(1,0,0\ldots,0)^T$. Then the set $\{\mathscr{X}_{1},\ldots, \mathscr{X}_{\ell}\}$  is said to be a T-orthogonal collection of tensors.  The collection is called T-orthonormal if $\alpha_{i}=1$, $i=1,\ldots,\ell$.

\noindent  Notice that the  $ \text{T-trace} $ of $(\mathscr{X}^{T}\star \mathscr{Y})$   can be expressed via the 3-mode product    as  follows: 
$$ ( \text{T-trace} {(\mathscr{X}^{T}\star \mathscr{Y})}\times_3 F_{n_3})^{(i)} =\text{trace} ( (\mathscr{X}\times_3 F_{n_3})^{(i)T}  (\mathscr{Y}\times_3 F_{n_3})^{(i)})\;\;, i=1,\ldots,n_3  $$

\begin{proposition}\label{proprinnerprodfrob}
	Let $\mathscr{A},\mathscr{B}$ and $\mathscr{C}$ be tensors of $ \mathbb{R}^{n_{1}\times s\times n_{3}}$ and ${\rm \bf a}\in  \mathbb{R}^{1\times 1\times n_{3}}$. Then the  Tubal-inner product satisfies the following properties
	\begin{enumerate}
		\item  $\langle \mathscr{A}, \mathscr{B}+\mathscr{C} \rangle_T$=$\langle \mathscr{A}, \mathscr{B}\rangle_T+\langle \mathscr{A},\mathscr{C} \rangle_T$.
		\item $\langle \mathscr{A}, {\rm \bf a}\divideontimes\mathscr{B}  \rangle_T$=${\rm \bf a}\star \langle \mathscr{A}, \mathscr{B}\rangle_T $.
		\item  
		$\langle \mathscr{A}, \mathscr{X}\star \mathscr{B}  \rangle_T=\langle \mathscr{X}^T\star\mathscr{A},   \mathscr{B}  \rangle_T,$ for $\mathscr{X} \in  \mathbb{R}^{n_{1}\times n_1\times n_{3}}.$
	\end{enumerate}
	\medskip
\end{proposition}

\begin{proof}
	For $i=1,\ldots,n_3$, we  have  
	 \begin{align*}
&( \text{T-trace} {(\mathscr{A}^{T}\star (\mathscr{B}+\mathscr{C}))}\times_3 F_{n_3})^{(i)}=\\
&  \text{trace} \left( (\mathscr{A}\times_3 F_{n_3})^{(i)T}  ((\mathscr{B}\times_3 F_{n_3})^{(i)}+(\mathscr{C}\times_3 F_{n_3})^{(i)})\right)= \\
& \text{trace} \left( (\mathscr{A}\times_3 F_{n_3})^{(i)T}  (\mathscr{B}\times_3 F_{n_3})^{(i)}+ (\mathscr{A}\times_3 F_{n_3})^{(i)T}(\mathscr{C}\times_3 F_{n_3})^{(i)}\right)=\\
	&(\text{T-trace} (\mathscr{A}^{T}\star \mathscr{B}+ \mathscr{A}^{T}\star \mathscr{C}) \times_3 F_{n_3})^{(i)}.	
		\end{align*}
	which shows the first property. 	The other properties could be easily shown in a similar way.
\end{proof}		

\subsection{The T-Diamond  product of third order tensors}

In this subsection,  we introduce  the T-Diamond product between two   tensors and give some algebraic properties.

\begin{definition}
	Let
		$\mathscr{A} =[\mathscr{A}_{1},\ldots,\mathscr{A}_{p}]$ where  $ \mathscr{A}_{i}$, $i =1,...,p,$ is an    $n_{1}\times s\times n_{3} $ tensor and let 
		$	\mathscr{B} =[\mathscr{B}_{1},\ldots,\mathscr{B}_{\ell}]$ where $  \mathscr{B}_{j}$, $j =1,...,\ell$ is an    $n_{1}\times s\times n_{3} $ tensor. 
		Then the T-diamond  product $\mathscr{A}^{T} \diamondsuit  \mathcal{B} $ is the  tensor of  size  $p \times  \ell \times n_{3}$ given by : 
		\begin{align*}
		(\mathscr{A}\diamondsuit \mathscr{B})   &=       \widetilde{(\mathscr{A}\diamondsuit \mathscr{B})}  \times_3 F_{n_3}^{-1},      \end{align*}
		\noindent where the  i-th  frontal  slice  of $  \widetilde{(\mathscr{A}\diamondsuit \mathscr{B})} $ is  given  by 
		\begin{align*}
		(\mathscr{A}\diamondsuit \mathscr{B})^{(i)} 
		&=        ({\mathscr {A}} \times_3 F_{n_3})^{(i)T}\diamond  (\mathscr{B}\times_3 F_{n_3})^{(i)},
		\end{align*}
		where 	   $\diamond $ is the diamond product between two matrices; for more details about the diamond product between two matrices,  see  \cite{Jbilousadaka}. 
\end{definition}
\medskip
\noindent The  T-diamond product  can be  computed  by  the  following   algorithm.

\begin{algorithm}[H]
	\begin{algorithmic}[1]
		\REQUIRE 	$\mathscr{A}\in {\mathbb R}^{n_{1}\times ps \times n_{3}} $ and   $\mathscr{B}\in {\mathbb R}^{n_{1}\times \ell s \times n_{3}}$  .
			\ENSURE: $\mathscr{A}^{T} \diamondsuit  \mathcal{B} $ is the  tensor of  size  $p \times  \ell \times n_{3}$.
		\STATE Set $\mathscr {\widetilde A}=\text{{\tt fft}}(\mathscr{ A},[\,],3) $   and   $ \mathscr {\widetilde B}=\text{{\tt fft}}(\mathscr{B},[\,],3) $,%
		\FOR{ $i=1,\dots,n_3$}
		\STATE  $
			(\mathscr{A}^T\diamondsuit\mathscr{B})^{( i)} =     ( {A}^{(i)T}\diamond   {B}^{(i)})$, 
		\ENDFOR
		\STATE $ (\mathscr{A}^T\diamondsuit\mathscr{B}) =\text{{\tt ifft}} (\widetilde{(\mathscr{A}^T\diamondsuit\mathscr{B})},[\,],3)$.
		 
	\end{algorithmic} 
	\caption{Tensor T-Diamond product}
	\label{diamondtensot}
\end{algorithm}
\noindent 	The next proposition gives some algebraic properties of the T-diamond product. 
\begin{proposition} \label{diamantpropos}		Let $\mathscr{A},\mathscr{B},\mathscr{C}\in {\mathbb R}^{n_{1}\times ps \times n_{3}}$, $\mathscr{D}\in {\mathbb R}^{n_{1}\times n_{1} \times n_{3}}$ and $\mathscr{L}\in {\mathbb R}^{p\times p \times n_{3}}$, We have the following proposals:
	\begin{enumerate}
		\item $ (\mathscr{A}+\mathscr{B})^{T}\diamondsuit \mathscr{C} = \mathscr{A}^{T}\diamondsuit \mathscr{C} +\mathscr{B}^{T}\diamondsuit \mathscr{C}$
		\item
		$\mathscr{A}^{T}\diamondsuit(\mathscr{B}+ \mathscr{C}) = \mathscr{A}^{T}\diamondsuit \mathscr{B} +\mathscr{A}^{T}\diamondsuit \mathscr{C}$   			
		\item
		$(\mathscr{A}^{T}\diamondsuit \mathscr{B})^{T} = \mathscr{B}^{T}\diamondsuit \mathscr{A} $   			
		\item
		$ (\mathscr{D}\star \mathscr{A})^{T}\diamondsuit \mathscr{B} = \mathscr{A}^{T}\diamondsuit (\mathscr{D}^{T}\star \mathscr{B})$   			
		\item  
		$ \mathscr{A}^{T}\diamondsuit (\mathscr{B}\star(\mathscr{L}\circledast \mathscr{I}_{ssn_3})) = (\mathscr{A}^{T}\diamondsuit \mathscr{B})\star \mathscr{L}$
	\end{enumerate}	
\end{proposition}
\medskip
\begin{proof}
	Obviously, the  results are derived  directly from the properties of the matrix-$\diamond$ product. 
	For $i=1,\ldots,n_3$ we have
	\begin{align*}
	& \left(	\mathscr{A}^{T}\diamondsuit (\mathscr{B}\star(\mathscr{L} \circledast \mathscr{I}_{ssn_3})))\times_3 F_{n_3}\right)^{(i)} =   \\
	& ({\mathscr {A}} \times_3 F_{n_3})^{(i)T} \diamond \left(  ({\mathscr {B}} \times_3 F_{n_3})^{(i)}) (  ({\mathscr {L}} \times_3 F_{n_3})^{(i)} 
	 \otimes   ({\mathscr {I}_{ssn_3}} \times_3 F_{n_3})^{(i)}\right)= \\
	&\left[  \left(({\mathscr {A}} \times_3 F_{n_3})^{(i)T} \diamond({\mathscr {B}} \times_3 F_{n_3})^{(i)}\right) ({\mathscr {L}} \times_3 F_{n_3})^{(i)}\right] =\\
	& \left((( \mathscr{A}^{T}\diamondsuit 
	\mathscr{B})\star\mathscr{L})\times_3 F_{n_3}\right)^{(i)}. 	\end{align*}
	\noindent Finally we  get :  $ \mathscr{A}^{T}\diamondsuit (\mathscr{B}\star(\mathscr{L}\circledast \mathscr{I}_{ssn_3})) = (\mathscr{A}^{T}\diamondsuit \mathscr{B})\star \mathscr{L}$.  The other results are obtained by following in  the same manner.  	
\end{proof}

\section{The tensor tubal global GMRES method}
\subsection{The tubal global QR factorization}
Next, we present the tubal-global Gram–-Schmidt process.
\medskip
\begin{definition}\label{inverstubscalar}
	Let ${\rm \bf z}\in {\mathbb R}^{1\times 1 \times n_{3}} $, then 
	the  tubal rank of  ${\rm \bf z}$ is the number of its
	non-zero Fourier coefficients. If the tubal-rank of ${\rm \bf z}$ is equal to  $n_3$, we say that  it is invertible and   we denote by $({\rm \bf z})^{-1}$	 the inverse of  $ {\rm \bf z}$ iff:  ${\rm \bf z}\star({\rm \bf z})^{-1}=({\rm \bf z})^{-1}\star{\rm \bf z}={\rm \bf e}$,  where ${\rm \bf e}$ is the tubal-fiber such that ${\rm unfold}({\rm \bf e})  =(1,0,0\ldots,0)^T$.
\end{definition}

\noindent First, we need to introduce a normalization algorithm. This means that   given a non-zero $\mathscr{A}\in {\mathbb R}^{n_{1}\times s \times n_{3}}$ we need to be able to decompose  the tensor $ \mathscr{A}$ as  
$$ \mathscr{A}={\rm \bf a}\divideontimes\mathscr{Q}= \mathscr{Q}\star({\rm \bf a} \circledast\mathscr{I}_{ssn_3}),$$  
where ${\rm \bf a}$ is invertible and  $\left\langle \mathscr{Q},\mathscr{Q}\right\rangle_T={\rm \bf e}$. We consider the following normalization algorithm. 

\begin{algorithm}[H]
	\begin{algorithmic}[1]
		\REQUIRE 	$\mathscr{A}\in {\mathbb R}^{n_{1}\times s \times n_{3}} $ and  $tol>0$ .
		\ENSURE:  $\mathscr{Q}\in {\mathbb R}^{n_{1}\times s \times n_{3}} $ and  ${\rm \bf a}\in {\mathbb R}^{1\times 1 \times n_{3}}  $ such that $\left\langle \mathscr{Q},\mathscr{Q}\right\rangle_T={\rm \bf e}$.
		\STATE Set $\mathscr{\widetilde{A}}=\text{{\tt fft}}(\mathscr{A},[\,],3) $,%
		\FOR{ $j=1,\dots,n_3$}
		\STATE ${\rm \bf a}^{(j)}= 				 	\text{trace}( {A}^{ (j)T}{A}^{(j)})=||{A}^{ (j)}||_F  $ 
		\STATE $~~~~$ \textbf{if}	{${\rm \bf a}^{(j)}<tol$ }   
		  
		  $~~~~~~~~$\textbf{Stop} 
		  
		$~~~~$ \textbf{else}   {${\mathscr {Q}}^{(j)}=\frac{ { A}^{(j)} }{{\rm \bf a}^{(j)}} $}
		
	   	$~~~~$ \textbf{end if}
		\ENDFOR
		\STATE $ \mathscr{Q}  =\text{{\tt ifft}}  (\mathscr{\widetilde{Q}} ,[\,],3)$, $ {\rm \bf a}  =\text{{\tt ifft}}(\widetilde{{\rm \bf a}} ,[\,],3)$.
			 
	\end{algorithmic} 
	\caption{A normalization algorithm (Normalization($\mathscr{A}$))}
	\label{normalization12}
\end{algorithm}

\noindent The next algorithm   summarizes the different steps defining the tubal-global QR factorisation of a tensor.
 	\begin{algorithm}[H]
 			\begin{algorithmic}[1]
			\REQUIRE $\mathscr{Z}=[ \mathscr{Z}_{1},\ldots, \mathscr{Z}_{k}]$ be an  $n\times ks\times n_3$ tensor   where $\mathscr{Z}_{j}$ is an $n\times s\times n_3$ tensor, for $j=1,\ldots,k $, $s<n$ .
		\STATE Set $[\mathscr{Q}_{1},{\rm \bf r}_{1,1}]= \text{ Normalization}(\mathscr{Z}_{1})$: using Algorithm \ref{normalization12}. 
		\FOR{ $j=2,\ldots,k$}
			\STATE $\mathscr{W}=     \mathscr{Q}_j$
			\FOR {$i=1,\ldots,j-1$}
				\STATE ${\rm \bf r}_{i,j}=\langle \mathscr{Q}_i, \mathscr{W} \rangle_T$
				\STATE $\mathscr{W}=\mathscr{W}-{\rm \bf r}_{i,j}\divideontimes\mathscr{Q}_i$	
			\ENDFOR
			\STATE $[\mathscr{Q}_{j},{\rm \bf r}_{j,j}]=  \text{ Normalization}(\mathscr{W} )$:  using Algorithm \ref{normalization12}.
		\ENDFOR
\end{algorithmic} 
\caption{The Tensor Tubal-Global QR decomposition} 
\label{TGQR}
\end{algorithm}
\medskip
	\begin{proposition}\label{T-QRpropos}
	Let    $ \mathscr{Z}=[ \mathscr{Z}_{1},\ldots, \mathscr{Z}_{k}]$ be an  $n\times ks\times n_3$ tensor   where $\mathscr{Z}_{j}$ is an $n\times s\times n_3$ tensor, for $j=1,\ldots,k$.  Then from Algorithm  \ref{TGQR}, the tensor $\mathscr{Z}$ can be factored as $$\mathscr{Z}=\mathscr{Q}\star (\mathscr{R}\circledast \mathscr{I}_{ssn_3}),$$ 
	where $\mathscr{Q}=[\mathscr{Q}_{1},\ldots,\mathscr{Q}_{k}] $ is an $(n\times ks\times n_3)$ T-orthonormal tensor satisfying	$\mathscr{Q}^{T}\diamondsuit \mathscr{Q}=\mathscr{I}_{kkn_3}$  and  $\mathscr{R}$ is an upper triangular $(k\times k\times n_3)$ tensor  
	(each  frontal  slice  of  $\mathscr{R}$  is  an  upper  triangular  matrix of  size $k\times k$) given by  
	$$ {\mathscr{   {R}}}=\left[ \begin{array}{*{20}{c}}
	{{\rm \bf r}_{1,1  }}&{{\rm \bf r}_{1,2  }}& \ldots & {{\rm \bf r}_{1,k  }}  \\
	& {{\rm \bf r}_{2,2}}&\cdots& {{\rm \bf r}_{2,k}}\\
	&   & \ddots & \vdots \\
	&    &       &    {\rm \bf r}_{k,k}
	\end{array}  \right]\in \mathbb{R}^{k\times k \times n_3} 
	$$ 
\end{proposition}
\medskip
\begin{proof}
This will be shown by induction on $k$. For $k=1$, we have from Line 2 of  Algorithm \ref{TGQR}: $ \langle \mathscr{Q}_1, \mathscr{Q}_1\rangle_T ={\rm \bf e}$.  Assume now that the result is true for some  $k$. Using the results of  Proposition \ref{proprinnerprodfrob}, we obtain 
	\begin{align*}
	( {\rm \bf r}_{k+1,k+1}) \star\langle \mathscr{Q}_{j},\mathscr{Q}_{k+1}\rangle_T &=\langle \mathscr{Q}_{j},  (\mathscr{W} - \sum_{i =1}^{k-1} {\rm \bf r}_{i,k } \divideontimes\mathscr{Q}_{i})  \rangle_T \\
	&=    \left( \langle  \mathscr{Q}_{j},\mathscr{W} \rangle_T   -(\sum_{i =1}^{k-1} {\rm \bf r}_{i,k}\star\langle \mathscr{Q}_{j},\mathscr{Q}_{i} \rangle_T \right)   \\
	&=    ( {\rm \bf r}_{j,k}- {\rm \bf r}_{j,k})   ={\rm \bf o},\;  j=1,\ldots,k,  
	\end{align*}
	where ${\rm \bf o}$ denotes de zeros tube fiber of size $(1\times 1\times n_3)$ which all
	his entries are equal to zeros. Then we get   $\mathcal{Q}^T\diamondsuit \mathcal{Q}=
	\mathscr{I}_{ k k n_{3}}$. \\
	Now consider  an  $n\times ks\times n_3$ tensor  $ \mathscr{Z}=[ \mathscr{Z}_{1},\ldots, \mathscr{Z}_{k}]$  where  $\mathscr{Z}_{j}$ is an $n\times s\times n_3$ tensor.  Then from Algorithm \ref{TGQR}, we have   $\mathscr{Z}_{j}=\displaystyle \sum_{i =1}^{j} {\tt r}_{i,j}\divideontimes \mathscr{Q}_{i}$ and the  j-th lateral slice      of $\mathscr{Z}$ is given by :
	\begin{align*}
	(\mathscr{Z})_j=\mathscr{Z}_j &= \displaystyle \sum_{i =1}^{j} {\rm \bf r}_{i,j}\divideontimes \mathscr{Q}_{i}	 \\
	&=\sum_{i =1}^{j}  \mathscr{Q}_{i}\star(( {\rm \bf r}_{i,j} )\circledast \mathscr{I}_{s,s,n_3})\\
	&=[ \mathscr{Q}_{1},\ldots, \mathscr{Q}_{j}]\star\left(  \left[ {\begin{array}{ {c}}
		{  {\rm \bf r}_{1,j}}  \\
		\vdots \\
		{  {\rm \bf r}_{j,j}}
		\end{array}} \right]\circledast \mathscr{I}_{s,s,n_3}\right). 
	\end{align*}
	Let  ${\mathscr{R}_j}= \left[ {\begin{array}{ {c}}
		{ {\rm \bf r}_{1,j}}  \\
		\vdots \\
		{  {\rm \bf r}_{j,j}}\\
		0\\
		\vdots\\
		0
		\end{array}}\right] \in {\mathbb R}^{k\times 1 \times n_{3}}  $ be  the j-th lateral slice of   the    ($k\times k\times n_3$)   tensor $\mathscr{  {R}} =[\mathscr{R}_1 ,\ldots, \mathscr{R}_{k}]$. Then we have the decomposition 
	$$ \mathscr{Z}_j=[ \mathscr{Q}_{1},\ldots, \mathscr{Q}_{j}]\star\left(\mathscr{R}_{j}   \circledast \mathscr{I}_{s,s,n_3}\right), \quad j=1,\ldots,k.
	$$
	Therefore, $\mathscr{Z}$ can  be factored as 
	$ \mathscr{Z}=\mathscr{Q}\star (\mathscr{R}\circledast \mathscr{I}_{s,s,n_3})$ where $[\mathscr{Q}_{1},\ldots,\mathscr{Q}_{k}] $ is an $(n\times ks\times n_3)$ T-orthonormal tensor 
	$\mathscr{Q}^{T}\diamondsuit \mathscr{Q}=\mathscr{I}_{kkn_3}$  and  $\mathscr{R}$ is an upper triangular $(k\times k\times n_3)$ tensor. 
\end{proof} 
\noindent 	Notice  that  $\mathscr{Q}^{T}\diamondsuit \mathscr{Z}=\mathscr{Q}^{T}\diamondsuit(\mathscr{Q}\star (\mathscr{R}\circledast \mathscr{I}_{ssn_3}))$, and by using the result 5) of   Proposition \ref{diamantpropos}, we get $\mathscr{Q}^{T}\diamondsuit \mathscr{Z}=(\mathscr{Q}^{T}\diamondsuit\mathscr{Q})\star\mathscr{R}=\mathscr{R} $.
	\subsection{The tensor  tubal-global Arnoldi process} In this section, we define the tubal-global Arnoldi process that could be considered as a generalisation of the global Arnoldi process defined in \cite{Jbilou} for matrices. In \cite{elguide1}, the authors  introduced the T-global Arnoldi process. The  main  difference between the  tubal-global Arnoldi  and  the T-global Arnoldi  is  that for tubal global Arnoldi process the    tensor Krylov global  subspace $\mathcal{\mathscr{K}}_m(\mathscr{A},\mathscr{V} )$ associated to the T-product is  as follows
\begin{equation}
\label{tr3}
\mathcal{\mathscr{K}}_m(\mathscr{A},\mathscr{V} )= {\rm Tspan}\{ \mathscr{V}, \mathscr{A} \star\mathscr{V},\ldots,\mathscr{A}^{m-1}\star\mathscr{V} \}\\
=\left\lbrace \mathscr{Z} \in \mathbb{R}^{n\times s \times n_3}, \mathscr{Z}= \sum_{i=1}^m \alpha_{i} \left(
\mathscr{A}^{i-1}\star\mathscr{V}\right) \right\rbrace 
\end{equation}
where $\alpha_{i}\in \mathbb{R},\; i=1,\ldots,m $;   
$\mathscr{A}\in \mathbb{R}^{n\times n \times n_3}$ and $\mathscr{V}\in \mathbb{R}^{n\times s \times n_3}$. \\ In  the  case of the tubal global Arnoldi process,  the tensor Tubal global Krylov  subspace  of order $m$  generated by $\mathscr{A}$ and $\mathscr{V}$ and denoted by $\mathscr{TK}^{g}_{m}(\mathscr{A},\mathscr{V})\subset   \mathbb{R}^{n\times s \times n_3}$  is defined by :
\begin{align}\label{ttgk}
\mathscr{TK}^{g}_{m}(\mathscr{A},\mathscr{V})&=\text{T-Span}\left\lbrace \mathscr{V},\mathscr{A}\star\mathscr{V},\mathscr{A}^2\star\mathscr{V},\ldots,\mathscr{A}^{m-1}\star\mathscr{V}\right\rbrace \\
&=\left\lbrace \mathscr{Z} \in \mathbb{R}^{n\times s \times n_3}, \mathscr{Z}= \sum_{i=1}^m {\rm \bf a}_{i}\divideontimes(
\mathscr{A}^{i-1}\star\mathscr{V}) \right\rbrace 
\end{align}
where ${\rm \bf a}_{i}\in \mathbb{R}^{1\times 1 \times n_3}$,  $\mathscr{A}^{i-1}\star\mathscr{V}=\mathscr{A}^{i-2}\star\mathscr{A}\star\mathscr{V}$, $i=2,\ldots,m$ and $\mathscr{A}^{0}$ is the identity tensor. The following tubal-global Arnoldi process  produces a T-orthogonormal basis of  $\mathscr{TK}^{g}_m(\mathscr{A},\mathscr{V})$. The algorithm is described as follows

\begin{algorithm}[H]
		\begin{algorithmic}[1]
		\REQUIRE	  $\mathscr{A}\in \mathbb{R}^{n\times n \times n_3}$, $\mathscr{V}\in \mathbb{R}^{n\times s \times n_3}$ and the positive integer m.
			\STATE Set $[\mathscr{V}_{1}, {\tt r}_{1,1}]=  \text{Normalization}(\mathscr{V})$
			\FOR{ $j=1,\ldots,m$}
				\STATE $\mathscr{W}=  \mathscr{A}\star   \mathscr{V}_j$
				\FOR {$i=1,\ldots,j$}
					\STATE $ {\rm \bf h}_{i,j}=\langle \mathscr{V}_i, \mathscr{W} \rangle_T$
					\STATE $\mathscr{W}=\mathscr{W}- {\rm \bf h}_{i,j}\divideontimes\mathscr{V}_i$	
				\ENDFOR
				\STATE $[\mathscr{V}_{j+1}, {\rm \bf h}_{j+1,j}]=  \text{Normalization}(\mathscr{W} )$.
			\ENDFOR
	\end{algorithmic}
	\caption{The Tensor Tubal-Global Arnoldi} \label{TTGA}
\end{algorithm}

\begin{proposition}
	Suppose that m steps of Algorithm \ref{TTGA} have been run. Then, the tensors  $\mathscr{V}_{1},\ldots,\mathscr{V}_{m}$, form a T-orthonormal basis of the Tubal-global Krylov  subspace $\mathscr{TK}^{g}_{m}(\mathscr{A},\mathscr{V})$.
\end{proposition}

\begin{proof}
	This will be shown by induction on $m$.  For $m=1$, we have from Line 2  of  Algorithm \ref{TGQR} the relation  $ \langle \mathscr{V}_1, \mathscr{V}_1\rangle_T ={\rm \bf e}$.  Assume now that the result is true for some $m$, then from Algorithm \ref{TGQR} and by using the results of  Proposition \ref{proprinnerprodfrob},  we get
	\begin{align*}
		({\rm \bf h}_{m+1,m}) \star\langle \mathscr{V}_{j},\mathscr{V}_{m+1}\rangle_T&= \langle \mathscr{V}_{j},({\rm \bf h}_{m+1,m})\divideontimes\mathscr{V}_{m+1}\rangle_T\\
		&=  \langle \mathscr{V}_{j},  (\mathscr{W} - \sum_{i =1}^{m}{\rm \bf h}_{i,m} \divideontimes\mathscr{V}_{i})\rangle_T     \\
		&=   \left(  \langle  \mathscr{V}_{j},\mathscr{W} \rangle_T    -(\sum_{i =1}^{m }{\rm \bf h}_{i,m}\star\langle \mathscr{V}_{j},\mathscr{V}_{i} \rangle_T) \right)    \\
		&=   ({\rm \bf h}_{j,m}-{\rm \bf h}_{j,m})    ={\rm \bf o},\, =1,\ldots,m.
		\end{align*}
		where ${\rm \bf o}$ denote de zeros tube fiber of size $(1\times 1\times n_3)$ which all
		his entries are equal to zeros.
	Furthermore, from Line 3(d) of Algorithm \ref{TTGA}, we immediately have $\langle \mathscr{V}_{m+1},\mathscr{V}_{m+1}\rangle_T={\rm \bf e}$. Therefore, the
	result is true for $m+1$ which completes the proof.
\end{proof}

	\noindent 	Let $\mathbb{V}_{m}  $ be  the   $(n \times sm \times n_{3})$ tensor whose  frontal slices are $ {\mathscr{V}}_{1},\ldots, {\mathscr{V}}_{m}$ and let    $\mathscr{\widetilde{H}}_{m}$ the $(m+1)\times m \times n_{3} $ hessenberg tensor defined by Algorithm \ref{TTGA} (Hessemberg tensor mean that every frontal slice of $\mathscr{  \widetilde{H}}_m$ is a  Hesemberg matrix)  and by $\mathscr{H}_{m}$ the tensor obtained from $\widetilde{\mathscr{ H}}_{m}$ by deleting its last horizontal slice.   $\mathscr{A}\star\mathbb{V}_{m}  $ is  the $(n \times (sm)\times n_{3})$ tensor whose  frontal slices are  $\mathscr{A}\star {\mathscr{V}}_{1},\ldots,\mathscr{A}\star {\mathscr{V}}_{m}$, respectively.  
	Using Definition \ref{bloctens0}, we can  set  
	\begin{align*}
	\mathbb{V}_{m}:=&\left[   {\mathscr{V}}_{1},\ldots, {\mathscr{V}}_{m}\right]\in  \mathbb{R}^{n\times   sm\times n_{3}}, \;  \\
	\mathscr{A}\star {\mathbb{V}}_{m}:=&[\mathscr{A}\star {\mathscr{V}}_{1},\ldots,\mathscr{A}\star {\mathscr{V}}_{m}]\in  \mathbb{R}^{n\times  s m\times n_{3}}\\
	\mathbb{V}_{m+1}:=&\left[  {\mathbb{V}}_m , {\mathscr{V}}_{m+1}\right]\in  \mathbb{R}^{n_1\times   (m+1)s \times n_{3}}. \\
		\end{align*}
The tensors $\mathscr{  \widetilde{H}}_m$ and $	\mathscr{H}_m$ are defined by 
			\begin{align*}
	\mathscr{  \widetilde{H}}_m=&\left[ \begin{array}{*{20}{c}}
	{\rm \bf h}_{1,1}&{{{\rm \bf h}_{1,2} }}&\cdot &{\rm \bf h}_{1,m}   \\
	{\rm \bf h}_{2,1}&{\rm \bf h}_{2,2}&\cdots&{\rm \bf h}_{2,m} \\
	&\ddots&\ddots& \vdots\\
	&   &{\rm \bf h}_{m,m-1} &{\rm \bf h}_{m,m}\\
	&    &       &  {\rm \bf h} _{m+1,m}
	\end{array}  \right] \in  \mathbb{R}^{(m+1)\times   m\times n_{3}}, \\
	\mathscr{H}_m=&\left[ \begin{array}{*{20}{c}}
	{\rm \bf h}_{1,1}&{{{\rm \bf h}_{1,2} }}&\cdot &{\rm \bf h}_{1,m}   \\
	{\rm \bf h}_{2,1}&{\rm \bf h}_{2,2}&\cdots&{\rm \bf h}_{2,m} \\
	&\ddots&\ddots& \vdots\\
	&   &{\rm \bf h}_{m,m-1} &  {\rm \bf h}_{m,m} 
	\end{array}  \right] \in  \mathbb{R}^{m\times   m\times n_{3}}.
	\end{align*} 
	\noindent The tesnor  $ {\mathscr H}_m$ can be obtained from 
	$ \mathscr{  \widetilde{H}}_{m}$ by deleting the  horizontal slice
	$$[{\rm \bf o}  ,\ldots,{\rm \bf o},{\rm \bf h} _{m+1,m}] ={\rm \bf h} _{m+1,m}\star \mathscr{E}_{m} \;\in  \mathbb{R}^{1 \times   m\times   n_{3}},$$ where ${\rm \bf o}$ denote de zeros tube fiber of size $(1\times 1\times n_3)$ which all
	his entries are equal to zeros, and  $\mathscr{E}_{m}=\left[{\rm \bf o}  ,\ldots,{\rm \bf o},{\rm \bf e}\right]\in \mathbb{R}^{1 \times   m \times   n_{3}} $ where ${\rm \bf e} $ the tube fiber such that ${\rm unfold}({\rm \bf e})  =(1,0,0\ldots,0)^T$.
	\noindent  The tensor $\mathscr{  \widetilde{H}}_{m}$ can be written as follows
	$$ 
	\mathscr{  \widetilde{H}}_{m}=\begin{bmatrix}
	\mathscr{H}_m \\
	{\rm \bf h}_{m+1,m}\star\mathscr{E}_{m} 
	\end{bmatrix}.$$

\noindent We can now state the following algebraic properties

\begin{proposition}\label{T-GlobalArnolproposit}
		Suppose that m steps of Algorithm \ref{TTGA} have been run. Then, the following statements hold:
		\begin{align*}
		\mathscr{A}\star\mathbb{V}_{m}=&\mathbb{V}_{m}\star (\mathscr{H}_{m} \circledast \mathscr{I}_{ssn_3}) + \mathscr{V}_{m+1}\star((
		{\tt h}_{m+1,m}\star \mathscr{E}_{m})\circledast   \mathscr{I}_{ssn_3}),\\
		\mathbb{V}_{m}^{T}\diamondsuit\mathcal{A}\star\mathbb{V}_{m}=&\mathscr{H}_{m}, \\
		\mathcal{A}\star\mathbb{V}_{m}=&\mathbb{V}_{m+1} \star(\mathscr{ \widetilde{H}}_m\circledast \mathscr{I}_{ssn_3}), \\
		\mathbb{V}_{m+1}^{T}\diamondsuit  \mathcal{A}\star\mathbb{V}_{m}=&\mathscr{ \widetilde{H}}_m,\\
		\mathbb{V}_{m}^{T} \diamondsuit\mathbb{V}_m=&\mathscr{I}_{ m m n_{3}}.
		\end{align*}
\end{proposition}

\begin{proof}
	We give a proof only for the third relation, the other relations could be obtained in the same way. From Algorithm \ref{TTGA}, we have $ \mathscr{A}\star\mathscr{V}_{j}=\displaystyle \sum_{i =1}^{j+1}{\rm \bf h}_{i,j}\divideontimes \mathscr{V}_{i}$ and by 
		using the fact that  $\mathscr{A}\star\mathbb{V}_{m}=[\mathscr{A}\star\mathscr{V}_{1},\ldots,\mathscr{A}\star\mathscr{V}_{m}]$,
		the j-th frontal slice      of $\mathscr{A}\star\mathbb{V}_{m}$ is given by  
		\begin{align*}
		(\mathscr{A}\star\mathbb{V}_{m})_j=\mathscr{A}\star\mathscr{V}_{j}&=\sum_{i =1}^{j+1}{\rm \bf h}_{i,j}\divideontimes \mathscr{V}_{i}	 \\
		&=\sum_{i =1}^{j+1}  \mathscr{V}_{i}\star(({\rm \bf h}_{i,j} )\circledast \mathscr{I}_{s,s,n_3}),
			\end{align*}
			also expressed as 
	$$	(\mathscr{A}\star\mathbb{V}_{m})_j=[ \mathscr{V}_{1},\ldots, \mathscr{V}_{j+1}]\star\left(  \left[ {\begin{array}{ {c}}
			{ {\rm \bf h}_{1,j}}  \\
			\vdots \\
			{ {{\rm \bf h}}_{j+1,j}}
			\end{array}} \right]\circledast \mathscr{I}_{s,s,n_3}\right). $$
		Let  ${\mathscr{H}_j}= \left[ {\begin{array}{ {c}}
			{ {\rm \bf h}_{1,j}}  \\
			\vdots \\
			{ {\rm \bf h}_{j+1,j}}\\
			0\\
			\vdots\\
			0
			\end{array}}\right] \in {\mathbb R}^{m+1\times 1 \times n_{3}}  $ be  the j-th lateral slice of  of the Hessemberg tensor  $\mathscr{ \widetilde{H}} =[\mathscr{H}_1 ,\ldots, \mathscr{H}_{m}]$.
		The we have  $$(\mathscr{A}\star\mathbb{V}_{m})_j=[ \mathscr{V}_{1},\ldots, \mathscr{V}_{j+1}]\star\left(\mathscr{H}_{j}   \circledast \mathscr{I}_{s,s,n_3}\right) \quad j=1,\ldots,m. 
		$$
		and the result follows. 
\end{proof}

\subsection{The tensor  tubal-global GMRES  method}

The tensor tubal-global GMRES method is  based on tubal-global Arnoldi process to build a T-orthonormal basis of the tensor   tubal global  Krylov subspace \eqref{ttgk}.  First, we need to introduce a new T-$\ell_{2}$ norm  defined in $ \mathbb{R}^{m\times 1\times n_{3}}$.
\begin{definition}
	Let $ \mathscr{X}\in \mathbb{R}^{m\times 1\times n_{3}}$, $ \mathscr{Y}\in \mathbb{R}^{m\times 1\times n_{3}}$, then  the  T-$\ell_{2}$ inner  product  of $ \mathscr{X}$ and $ \mathscr{Y}$   is defined by 
	\begin{equation}
	\left\langle \mathscr{  {X}}, \mathscr{Y}\right\rangle _{T_{\ell_{2}}}=  \frac{1}{ {n_3}} \left( \sum_{i=1}^{n_3}   (\mathscr{  {X}}\times_3 F_{n_3} )^{(i)T}   (\mathscr{  {Y}} \times_3 F_{n_3} )^{(i)} \right ).   	\end{equation}
	The associated T-$\ell_{2}$ norm is defined by 
	\begin{equation}
	||\mathscr{  {Y}}||_{T_{\ell_{2}}}=\frac{1}{\sqrt{n_3}}\left( \sum_{i=1}^{n_3}|| (\mathscr{  {Y}} \times_3 F_{n_3} )^{(i)}  ||_2^2 \right)^{\frac{1}{2}}, 
	\end{equation}
	where   $||.||_{2}$ denotes the usual vector $\ell_{2}$-norm. In the following proposition, we give some algebraic relations associated to  T-$\ell_{2}$ inner  product.
\end{definition}

\begin{proposition}\label{tl2proprinnerprodfrob}
	Let $\mathscr{A},\mathscr{B},\mathscr{C}\in  \mathbb{R}^{m\times 1\times n_{3}}$ and $  \alpha\in  \mathbb{R} $, then the T-$\ell_{2}$-inner product satisfies the following direct algebraic  properties. 
	\begin{enumerate}
		\item  $\langle \mathscr{A}, \mathscr{B}+\mathscr{C} \rangle_{T_{\ell_{2}}}$=$\langle \mathscr{A}, \mathscr{B}\rangle_{T_{\ell_{2}}}+\langle \mathscr{A},\mathscr{C} \rangle_{T_{\ell_{2}}}$.
		\item $\langle \mathscr{A}, \alpha \mathscr{B}  \rangle_{T_{\ell_{2}}}$=$\alpha  \langle \mathscr{A}, \mathscr{B}\rangle_{T_{\ell_{2}}} $.
		\item 
		$\langle \mathscr{A}, \mathscr{X}\star \mathscr{B}  \rangle_{T_{\ell_{2}}}=\langle \mathscr{X}^T\star\mathscr{A},   \mathscr{B}  \rangle_{T_{\ell_{2}}},$  for $\mathscr{X} \in  \mathbb{R}^{n_{1}\times n_1\times n_{3}}$.  
	\end{enumerate}
\end{proposition}
\noindent The next property is important for the simplification of some minimisation properties to be used later.

\begin{proposition}\label{propoortho}
	Let $ \mathscr{  {Y}}\in \mathbb{R}^{m\times 1\times n_{3}}$ and $\mathscr{V} \in \mathbb{R}^{n\times ms\times n_{3}}$  such that  $\mathscr{V} ^{T} \diamondsuit \mathscr{V} =\mathscr{I}_{ m m n_{3}}$. Then
	\begin{align*}
	||\mathscr{  {V}}\star(\mathscr{  {Y}}\circledast \mathscr{I}_{ssn_3})||_F=||\mathscr{  {Y}}||_{T_{\ell_{2}}}.
	\end{align*}
\end{proposition}

\begin{proof}
	We have  \begin{align*}
	||\mathscr{  {V}}\star(\mathscr{  {Y}}\circledast \mathscr{I}_{ssn_3})||_F^2&=\frac{1}{ {n_3}}(\sum_{i=1}^{n_3}||(\mathscr{  {V}}\times_3 F_{n_3} )^{(i)}  \left(  (\mathscr{  {Y}} \times_3 F_{n_3} )^{(i)} \otimes   (\mathscr  {I}_{ssn_3} \times_3 F_{n_3} )^{(i)}\right)  ||^2_F) \\
	&=\frac{1}{ {n_3}}(\sum_{i=1}^{n_3}(|| (\mathscr{  {Y}} \times_3 F_{n_3} )^{(i)} ||_2^2)) \\
	&=||\mathscr{  {Y}}||_{T_{\ell_{2}}}^2.
	\end{align*}
\end{proof}
\noindent Next, we will see how to define the tubal-global GMRES.  Consider the linear system   of tensor equations
\begin{equation}\label{syslintens0}
\mathscr{A}\star \mathscr{X}=\mathscr{B}
\end{equation}  
where $\mathscr{A}\in \mathbb{R}^{n\times n \times n_3}$ assumed to be nonsingular, $\mathscr{B}$, $\mathscr{X} \in \mathbb{R}^{n\times s \times n_3}$ with  $ s \ll n$. If $n_3=1$ then the problem  \eqref{syslintens0}  reduces to a  multiple linear systems of $s$ equations . 
Let $\mathscr{  {X}}_{0}\in \mathbb{R}^{n\times s\times n_{3}}$ be an arbitrary initial guess   with  the associated  residual tensor 
$\mathscr{R}_{0}=\mathscr{B}-\mathscr{A}\star \mathscr{X}_0$.    The aim of the tensor  tubal-global GMRES method is to find, at some step $m$,   an approximation $\mathscr{X}_{m}$ of the solution $\mathscr{X}^*$ of  the problem \eqref{syslintens0} as follows 
\begin{align}\label{GLgmrescondition0}
\mathscr{X}_{m}-\mathscr{X}_{0}\in \mathscr{TK}^{g}_{m}(\mathscr{A},\mathscr{R}_0), 
\end{align}
with  \begin{align}\label{GLgmresconditionminim} 
||\mathscr{R}_m ||_{F}=\displaystyle \min_{\mathscr{X}-\mathscr{X}_{0}\in \mathscr{TK}^{g}_{m}(\mathscr{A},\mathscr{R}_0)}\left\lbrace ||\mathscr{B}-\mathscr{A}\star \mathscr{X}||_F
\right\rbrace.  \end{align}
\noindent 	From  \eqref{GLgmrescondition0}, $\mathscr{X}_{m}$ is expressed as    $\mathscr{X}_{m}=\mathscr{X}_{0}+ \mathbb{V}_{m}\star(\mathscr{  {Y}_{m}}\circledast \mathscr{I}_{ssn_3} )$ with $ \mathscr{  {Y}}=\left[ {\begin{array}{ {c}}
		{ {{\rm \bf y}}_{1 }}  \\
		\vdots \\
		\vdots\\
		{ {{\rm \bf y}}_{m }}\\
		\end{array}}\right]\in \mathbb{R}^{m\times 1\times n_{3}}$. Therefore  the minimisation  problem \eqref{GLgmresconditionminim} is equivalent to 
	\begin{align}\label{lowy}
	||\mathscr{R}_m ||_{F}&=  \displaystyle \min_{ \mathscr{  {Y}}\in \mathbb{R}^{m\times 1\times n_{3}}}
	||\mathscr{R}_{0}-\mathscr{A}\star\mathbb{V}_{m}\star( \mathscr{Y}\circledast \mathscr{I}_{ssn_3} )||_F.
	\end{align}  
	Using   Proposition \ref{propoortho} and   Proposition \ref{kronprop1} and Step 2 of  Algorithm \ref{TTGA}, we get 
	\begin{align*}
&	||\mathscr{R}_{0}-\mathscr{A}\star\mathbb{V}_{m}\star( \mathscr{Y}\circledast \mathscr{I}_{ssn_3} )||_F\\
&=|| {\rm \bf r}_{1,1} \divideontimes\mathscr{V}_{1}   -(\mathbb{V}_{m+1} \star(\mathscr{ \widetilde{H}}_m\circledast \mathscr{I}_{ssn_3}))\star( \mathscr{Y}\circledast \mathscr{I}_{ssn_3} ))||_F\\
	&=||\mathscr{V}_{1}\star({\rm \bf r}_{1,1}\circledast \mathscr{I}_{ssn_3})-(\mathbb{V}_{m+1} \star(\mathscr{ \widetilde{H}}_m\circledast \mathscr{I}_{ssn_3}))\star( \mathscr{Y}\circledast \mathscr{I}_{ssn_3} ))||_F\\
	&=||\mathbb{V}_{m+1}\star  (\mathscr{E}_1^{(m+1)}\circledast \mathscr{I}_{ssn_3})\star({\rm \bf r}_{1,1}\circledast \mathscr{I}_{ssn_3})-(\mathbb{V}_{m+1} \star ((\mathscr{ \widetilde{H}}_m\star\mathscr{Y})\circledast \mathscr{I}_{ssn_3} ))||_F\\
	&=||\mathbb{V}_{m+1}\star(( \mathscr{E}_1^{(m+1)}\star {\rm \bf r}_{1,1} )\circledast \mathscr{I}_{s,s,n_3})   -( (\mathscr{ \widetilde{H}}_m\star\mathscr{Y})  \circledast \mathscr{I}_{ssn_3} ))||_F\\
	&= ||\mathbb{V}_{m+1}\star(  \mathscr{E}_1^{(m+1)}\star {\rm \bf r}_{1,1} -( \mathscr{ \widetilde{H}}_m\star\mathscr{Y}))  \circledast \mathscr{I}_{ssn_3} )||_F\\
	&=|| (  \mathscr{E}_1^{(m+1)}\star {\rm \bf r}_{1,1}  -( \mathscr{ \widetilde{H}}_m\star\mathscr{Y}))   ||_{T_{\ell_{2}}},
	\end{align*}
	where $\mathscr{E}_1^{(m+1)}= \left[ {\begin{array}{ {c}}
		{\rm \bf e}  \\
		{\rm \bf o}\\
		\vdots\\
		{\rm \bf o}
		\end{array}}\right]\in \mathbb{R}^{m+1 \times 1 \times   n_{3}}$.
	Therefore, the tensor $	\mathscr{Y}_m$ solving the minimisation problem \eqref{lowy} is given by 
	\begin{equation}\label{Gmressol}
	\mathscr{Y}_m=  \arg \displaystyle  \min_{ \mathscr{  {Y}}\in \mathbb{R}^{m\times 1\times n_{3}}}|| \mathscr{E}_1^{(m+1)}\star {\tt r}_{1,1} -( \mathscr{ \widetilde{H}}_m\star\mathscr{Y})   ||_{T_{\ell_{2}}}. 
	\end{equation}
	The approximate solution is given by
	\begin{equation}\label{solutdegmresxm}
	\mathscr{X}_{m}=\mathscr{X}_{0}+ \mathbb{V}_{m}\star(\mathscr{  {Y}}_{m}\circledast \mathscr{I}_{ssn_3} ) 
	\end{equation}

\subsection{Implementation of tensor  tubal-global GMRES  method}
For the implementation of the tensor tubal-global GMRES algorithm, , we use the 
	tensor QR decomposition to solve the minimization problem (\ref{Gmressol}). The  Tensor QR decomposition is based on the application of the QR matrix decomposition to  each sub-block of  the obtained  block diagonal matrix in the Fourier domain. In other words, for  $\mathscr{  {F}}\in \mathbb{R}^{n\times m \times n_3}$  we have  $\mathscr{F}=\mathscr{  {Q}}\star\mathscr{  {R}}$ 
	where  $\mathscr{  {Q}}$ is an $n\times n \times  n_{3}$ orthogonal tensor   $\mathscr{Q}^{T}\star \mathscr{Q}   =\mathscr{Q} \star \mathscr{Q}^T=\mathscr{I}_{ n n n_{3}}$
	and $\mathscr{  {R}}\in \mathbb{R}^{n\times m \times n_3}$ triangular tensor.\\ The different steps are summarized as follows

\begin{algorithm}[H]
		\begin{algorithmic}[1]
	\REQUIRE $\mathscr{F}\in {\mathbb R}^{n\times m \times n_{3}} $.
	\ENSURE  Orthogonal tensor   $\mathscr{Q}\in {\mathbb R}^{n\times n \times n_{3}}$ and  triangular tensor $\mathscr{R}\in {\mathbb R}^{n\times m \times n_{3}} $ .
		\STATE Set $\mathscr{\tilde{F}}={\tt fft}(\mathscr{ F},[\,],3),$   
			\FOR {$i=1,\ldots,n_3$}
		\STATE $	 [ { {Q}}^{( i)} , { {R}}^{( i)}] =   \text{QR}  ( {F}^{( i)} )$, \quad (matrix QR decomposition )
			\ENDFOR
			\STATE $ \mathscr{Q} = {\tt ifft} (\tilde{\mathscr{Q}} ,[\,],3)$, $\mathscr{R} = {\tt ifft} (\tilde{\mathscr{R}} ,[\,],3)$	
	\end{algorithmic}
\caption{Tensor T-QR decomposition}\label{TQRsimple}
\end{algorithm}
\noindent To solve the problem \eqref{Gmressol}, we need the result of the following conservation-norm property.   	
	\begin{proposition}\label{QRvial2norm}
		Let $ \mathscr{  {Y}}\in \mathbb{R}^{m\times 1\times n_{3}}$ and $\mathscr{  {Q}} \in \mathbb{R}^{m\times m\times n_{3}}$ such that $\mathscr{Q}^{T} \star\mathscr{Q}=\mathscr{Q} \star\mathscr{Q}^T=\mathscr{I}_{ m m n_{3}}$. Then 
		\begin{align*}
		||\mathscr{  {Q}}\star\mathscr{  {Y}} ||_{T_{\ell_{2}}}=||\mathscr{  {Y}}||_{T_{\ell_{2}}}	\end{align*}
\end{proposition}
\medskip
\begin{proof} 
	The proof is a direct application of the T-${\ell}_{2}$ norm. 
\end{proof}
\medskip

\noindent Now, we  apply the T-QR decomposition to $ \mathscr{ \widetilde{H}}_m$, and by  using Proposition \ref{QRvial2norm},  we get  
\begin{align*}
|| \mathscr{E}_1^{(m+1)}\star {\tt r}_{1,1}  -( \mathscr{ \widetilde{H}}_m\star\mathscr{Y})   ||_{T_{\ell_{2}}}&=||\mathscr{Q}_m^T\star (\mathscr{E}_1^{(m+1)}\star {\tt r}_{1,1}  -( \mathscr{ \widetilde{H}}_m\star\mathscr{Y}))   ||_{T_{\ell_{2}}}\\
&=||\mathscr{Q}_m^T\star (\mathscr{E}_1^{(m+1)}\star {\tt r}_{1,1})  -\mathscr{Q}_m^T\star ( \mathscr{ \widetilde{H}}_m\star\mathscr{Y})   ||_{T_{\ell_{2}}}\\
&=||\widetilde{\mathscr{G}}_m   -\widetilde{\mathscr{R}}_m \star\mathscr{Y})   ||_{T_{\ell_{2}}}, 
\end{align*}
where $\widetilde{\mathscr{G}}_m=\mathscr{Q}_m^T\star (\mathscr{E}_1^{(m+1)}\star {\rm \bf r}_{1,1}) $ and $\widetilde{\mathscr{R}}_m =\mathscr{Q}_m^T\star  \mathscr{ \widetilde{H}}_m$.\\

\noindent The next property gives the solution $\mathscr{Y}_m$ of the minimisation problem \eqref{Gmressol}.

\begin{proposition}
	Let $\widetilde{\mathscr{G}}_m$ and $\widetilde{\mathscr{R}}_m$  given as 
	\begin{equation}\label{grm}
	\widetilde{\mathscr{R}}_m =\mathscr{Q}_m^T\star  \mathscr{ \widetilde{H}}_m, \; {\rm and}\; 
	\widetilde{\mathscr{G}}_m =\mathscr{Q}_m^T\star ( \mathscr{E}_1^{(m+1)}\star {\rm \bf r}_{1,1})=\left[ {\begin{array}{ {c}}
		{ {{\rm \bf g}}_{1 }}  \\
		\vdots \\
		{ {{\rm \bf g}}_{m+1 }}\\
		\end{array}}\right] \in {\mathbb R}^{m+1\times 1 \times n_{3}}.
	\end{equation}
	
	\noindent The solution $\mathscr{Y}_m=  \text{arg } \displaystyle  \min_{ \mathscr{  {Y}}\in \mathbb{R}^{m\times 1\times n_{3}}}|| \mathscr{E}_1^{(m+1)}\star {\rm \bf r}_{1,1}  -( \mathscr{ \widetilde{H}}_m\star\mathscr{Y}))   ||_{T_{\ell_{2}}}
		$ is given by  solving  the following triangular tensor problem
		\begin{equation}\label{Triangsup}
		\mathscr{R}_m \star\mathscr{Y}_m=  \mathscr{G}_m,
		\end{equation}
		where $\mathscr{R}_m=\widetilde{\mathscr{R}}_m(1:m,:,:)$ the tensor  obtained by deleting the last horizontal slice of $\widetilde{\mathscr{R}}$  and $\mathscr{G}_m=\left[ {\begin{array}{ {c}}
			{ {{\rm \bf g}}_{1 }}  \\
			\vdots \\
			{ {{\rm \bf g}}_{m }}\\
			\end{array}}\right] \in {\mathbb R}^{m\times 1 \times n_{3}}  $.
\end{proposition}
\begin{proof} We have 
	\begin{align*}
		|| \mathscr{E}_1^{(m+1)}\star {\rm \bf r}_{1,1}  -( \mathscr{ \widetilde{H}}_m\star\mathscr{Y})   ||_{T_{\ell_{2}}}&=||\mathscr{Q}_m^T\star (\mathscr{E}_1^{(m+1)}\star {\rm \bf r}_{1,1}  -( \mathscr{ \widetilde{H}}_m\star\mathscr{Y}))   ||_{T_{\ell_{2}}}\\
		&=||\mathscr{Q}_m^T\star (\mathscr{E}_1^{(m+1)}\star {\rm \bf r}_{1,1})  -\mathscr{Q}_m^T\star ( \mathscr{ \widetilde{H}}_m\star\mathscr{Y})   ||_{T_{\ell_{2}}}\\
		&=||\widetilde{\mathscr{G}}_m   -\widetilde{\mathscr{R}}_m \star\mathscr{Y})   ||_{T_{\ell_{2}}}\\
		&=  ||	{ {{\rm \bf g}}_{m+1 }}   ||^2_{T_{\ell_{2}}}+
		|| \mathscr{G}_m   - \mathscr{R}_m \star\mathscr{Y})   ||^2_{T_{\ell_{2}}}.
		\end{align*}
		Assuming that $\mathscr{R}_m$ is invertible,  then  $\mathscr{Y}_m$  solves the  triangular tensor problem \eqref{Triangsup}.
\end{proof}

\noindent  In the  following, we introduce the tubal-back substitution method  for  solving the  equation (\ref{Triangsup}). This 
	method follows the same steps as for   matrix backward substitution, where the tube fibers (3-mode fibers), lateral slices and T-product play the role of  scalars, vectors and matrix product, respectively. In other words,  the solution of  the  triangular  tensor system   (\ref{Triangsup})  can be obtained as follows 
	$$
	\left[ \begin{array}{*{20}{c}}
	{{\rm \bf r}_{1,1  }}&{{\rm \bf r}_{1,2  }}& \ldots & {{\rm \bf r}_{1,k  }}  \\
	& {{\rm \bf r}_{2,2}}&\cdots& {{\rm \bf r}_{2,k}}\\
	&   & \ddots & \vdots \\
	&    &       &    {\rm \bf r}_{m,m}
	\end{array}  \right]  \star \left[ {\begin{array}{ {c}}
		{ {{\rm \bf y}}_{1 }}  \\
		\vdots \\
		\vdots\\
		{ {{\rm \bf y}}_{m }}\\
		\end{array}}\right]= \left[ {\begin{array}{ {c}}
		{ {{\rm \bf g}}_{1 }}  \\
		\vdots \\
		\vdots \\
		{ {{\rm \bf g}}_{m }}\\
		\end{array}}\right]  $$
with ${{\rm \bf y}}_{m }= ({\rm \bf r}_{m,m})^{-1}\star {\tt g}_{m}$ 
and  $$ { {{\rm \bf y}}_{i }}=({\rm \bf r}_{i,i})^{-1}\star({ {{\rm \bf g}}_{i }}-\sum_{j=i+1}^{m}{\rm \bf r}_{i,j}\star {\rm \bf y}_{j})\;\;\text{ for } i=m-1,m-2,\ldots,1,$$ where  $({\rm \bf r}_{i,i})^{-1}$ stands for the inverse of  the tube fiber $ {\rm \bf r}_{i,i}$ (Definition \ref{inverstubscalar} ) for $ i=m,m-1,\ldots,1  $.\\

\noindent The whole steps of the tensor tubal-global GMRES algorithm are summarized in the following algorithm.

\begin{algorithm}[H]
		\begin{algorithmic}[1]
		\REQUIRE  $\mathscr{A}\in \mathbb{R}^{n\times n \times n_3}$, $\mathscr{V}$, $\mathscr{B}$, $\mathscr{X}_{0}\in \mathbb{R}^{n\times s \times n_3}$, the maximum number of iterations  $\text{Iter}_{\text{max}} $,   the restart parameter $m$  and a tolerance  $tol>0$.
		\ENSURE	  $\mathscr{X}_m\in \mathbb{R}^{n\times s \times n_3}$ the approximate solution of  (\ref{syslintens0}).  
		\STATE  Compute $\mathscr{R}_{0}=\mathscr{B}-\mathscr{A}\star\mathscr{X}_{0} $.
		\FOR {$k=1,\ldots,\text{Iter}_{\text{max}}$}
 	\STATE   Apply Algorithm \ref{TTGA} to compute  $\mathbb{V}_{m}$ and  $\mathscr{ \widetilde{H}}_m$.
	 			\STATE Compute the T-QR decomposition of  $\mathscr{ \widetilde{H}}_m$ using Algorithm \ref{TQRsimple} . 
 	\STATE  Compute  $\mathscr{R}_m$ and $\mathscr{G}_m$
 		using the relations (\ref{grm}) .
	 	\STATE  Solve  the triangular tensor system (\ref{Triangsup}) to obtain  $\mathscr{Y}_m$ .
	 	\STATE  Compute  $\mathscr{X}_{m}=\mathscr{X}_{0}+ \mathbb{V}_{m}\star(\mathscr{Y}_m\circledast \mathscr{I}_{ssn_3} )$
\STATE	$~~~~$ \textbf{if}$||\mathscr{R}_{m}||_F<tol$, 

$~~~~~~~~$\textbf{Stop} 

	 	\STATE $~~~~$\textbf{ else }  $\mathscr{X}_{0}=\mathscr{X}_{m}$ and go to Step 1.
	 	
	 	$~~~~$ \textbf{end if}
		 \ENDFOR 
	\end{algorithmic}
\caption{The Tensor Tubal-Global GMRES (m)} 
 \label{TTGGMRES}
\end{algorithm}

\section{Tensor   tubal-global Golub Kahan algorithm}
\noindent Instead of using the tensor global Arnoldi process to generate a basis for the projected  subspace, we can use the tensor tubal global Lanczos process. Here, we will use the tensor Golub Kahan algorithm related to the T-product. We notice  that we already defined in \cite{elguide2} another version of the tensor Golub Kahan algorithm for ill-posed problems  with applications to color image processing.\\
Consider the least squares  problem of  tensors
\begin{equation}\label{syslintens01}
\displaystyle \min_{\mathscr{X} } \Vert \mathscr{A}\star \mathscr{X} -  \mathscr{B}\Vert_F
\end{equation}  
where   $\mathcal {A} \in \mathbb{R}^{n_{1}\times n_{2}\times n_{3}}$  and    $\mathcal {B} \in \mathbb{R}^{n_{1}\times  s \times n_{3}}$.  The tensor tubal-global Golub Kahan bidiagonalization algorithm (Algorithm \ref{TTG-GK}) produces a T-orthogonormal basis   $\mathbb{U}_{k}=[\mathscr{  {U}}_{1},\ldots,\mathscr{  {U}}_{k},\mathscr{  {U}}_{k+1}]  \in \mathbb{R}^{n_{1}\times (k+1)s\times n_{3}} $ and $\mathbb{V}_{k}=[\mathscr{  {V}}_{1},\ldots,\mathscr{  {V}}_{k}]   \in \mathbb{R}^{n_{2}\times ks\times n_{3}} $     of the tensor Krylov subspace  $\mathscr{TK}^{g}_k(\mathscr{A}\star \mathscr{A}^T,\mathscr{A}\star \mathscr{A}^T\star \mathscr{B})$ and $\mathscr{TK}^{g}_k(\mathscr{A}^T\star \mathscr{A},\mathscr{A}^T\star \mathscr{B})$, respectively.  The algorithm is given as follows

\begin{algorithm}[H]
\begin{algorithmic}[1]
\REQUIRE The tensors $\mathscr {A}$, $\mathcal {B}$,   and an integer $m$.
		\STATE  	Set $\mathscr{V}_{0}=0   \in \mathbb{R}^{n_{1} \times s \times n_{3}} $ and $[\mathscr{U}_{1},{\rm \bf a}_{1 }]= \text{Normalization}(\mathscr{B})$.
		\FOR{ $j=1,\ldots,k$}
			\STATE $\widetilde {\mathscr {V}}= \mathscr {A}^T \star \mathscr {U}_{j}-{\rm \bf a}_{j} \divideontimes \mathscr {V}_{j-1}$
			\STATE Set $[\mathscr{V}_{j},{\rm \bf b}_{j }]=  Normalization(\widetilde {\mathscr {V}})$	 
			\STATE $\widetilde {\mathscr {U}}=\mathscr {A} \star \mathscr {V}_j-{\rm \bf b}_j\divideontimes \mathscr {U} _{j}$
			\STATE $[\mathscr{U}_{j+1},{\rm \bf a}_{j+1 }]=  \text{Normalization}(\widetilde {\mathscr {U}})$	 
		\ENDFOR
	\end{algorithmic}
		\caption{The Tensor Tubal-Global Golub-Kahan algorithm}
		\label{TTG-GK}
\end{algorithm}

\medskip 

\noindent Let $\widetilde{\mathscr{   {C}}}_k$ be the upper bidiagonal $(k+1) \times k\times n_3$ tensor 
$$ \widetilde{\mathscr{   {C}}}_k=\left[ \begin{array}{*{20}{c}}
{{{\rm \bf b}_1  }}&{{ }}& &   \\
{\rm \bf a}_{2}&{{{\rm \bf b}_2}}&\ddots& \\
&\ddots&\ddots& \\
&   & {\rm \bf a}_{k}  &  {\rm \bf b}_{k}\\
&    &       &    {\rm \bf a}_{k+1}
\end{array}  \right] 
$$
and let $ {\mathscr{   {C}}}_k$ be the $(k \times k\times n_3)$ tensor obtain   by deleting  the last horizontal slice of  $\widetilde{\mathscr{   {C}}}_k$ where the ${\rm \bf a}_i$'s and ${\rm \bf b}_i$'s are fibers then we have the following results.

\begin{proposition}
	\noindent 	The tensors $\mathbb{V}_{k}=[\mathscr{  {V}}_{1},\ldots,\mathscr{  {V}}_{k}]   \in \mathbb{R}^{n_{2}\times ms\times n_{3}} $ and  $\mathbb{U}_{k}=[\mathscr{  {U}}_{1},\ldots,\mathscr{  {U}}_{k}]  \in \mathbb{R}^{n_{1}\times ks\times n_{3}} $  given by Algorithm\ref{TTG-GK},  have orthogonal   tensors
	$ \mathscr{  {V}}_{i}\in \mathbb{R}^{n_{2}\times  s\times n_{3}}$ and $\mathscr{  {U}}_{i}\in \mathbb{R}^{n_{1}\times  s\times n_{3}}$, respectively, i.e. \begin{align*}
	\left\langle {\mathscr{V}_{i}, \mathscr{V}_{j}} \right\rangle_T =\left\langle {\mathscr{U}_{i}, \mathscr{U}_{j}} \right\rangle_T
	=\begin{cases}
	{\rm \bf e}&i= j \\
	0&i\neq j.
	\end{cases}.
	\end{align*}
	
\end{proposition}

\begin{proof}
	This will be shown by induction on j. For  $j=1$,  Algorithm \ref{TTG-GK}   shows that $\langle \mathscr{V}_1, \mathscr{V}_1\rangle_T ={\rm \bf e}$ and  $ \langle \mathscr{U}_1, \mathscr{U}_1\rangle_T ={\rm \bf e}$.  Assume now that the result is true for some j. Then, from Algorithm \ref{TTG-GK} and using the results of  Proposition \ref{proprinnerprodfrob}, we conclude that $$\left\langle {\mathscr{V}_{i}, \mathscr{V}_{j}} \right\rangle_T=\left\langle {\mathscr{U}_{i}, \mathscr{U}_{j}} \right\rangle_T
	=\begin{cases}
	{\rm \bf e}&i= j \\
	0&i\neq j
	\end{cases},\; j=1,\ldots,k.$$
\end{proof}

\begin{proposition} 
		The tensors produced by the tensor tubal-global Golub-Kahan algorithm satisfy the following relations
		\begin{eqnarray} \label{equa201}
		\mathscr{A} \star \mathbb{V}_k& = &\mathbb{U}_{k+1} \star({\widetilde {\mathscr{   {C}}}}_k\circledast \mathscr{I}_{ssn_3})    \\
		& = &\mathbb{U}_k\star({  {\mathscr{   {C}}}}_k\circledast \mathscr{I}_{ssn_3})+\mathscr{U}_{k+1} \star((
		{\tt a}_{k+1 }\star \mathscr{E}_{k})\circledast   \mathscr{I}_{ssn_3}), \; {\rm and } \\
		\mathscr{A}^{T}\star\mathbb{U}_{k}& = &\mathbb{V}_k  \star({ {\mathscr{   {C}}}}_k^T\circledast \mathscr{I}_{ssn_3}) \\  
		\mathscr{B}&=&\mathbb{U}_{k+1}\star((\mathscr{E}_1^{(k+1)}\star {\rm \bf a}_{1 }  )\circledast \mathscr{I}_{ssn_3} ). 
		\end{eqnarray}
	\end{proposition}
	where     $\mathscr{E}_{k}=\left[{\rm \bf o}  ,\ldots,{\rm \bf o},{\rm \bf e}\right]\in \mathbb{R}^{1 \times   k \times   n_{3}} $,  $\mathscr{E}_1^{(k+1)}= \left[ {\begin{array}{ {c}}
		{\rm \bf e}  \\
		{\rm \bf o}\\
		\vdots\\
		{\rm \bf o}
		\end{array}}\right]\in \mathbb{R}^{k+1 \times 1 \times   n_{3}}$, where   ${\rm \bf e} $ the tube fiber such that ${\rm unfold}({\rm \bf e})  =(1,0,0\ldots,0)^T$  and   ${\rm \bf o}$ denotes de zeros tube fiber of size $(1\times 1\times n_3)$ whit all entries are equal to zero.

\begin{proof}
	The proofs come directly from the different steps of Algorithm 
	\ref{TTG-GK}.
\end{proof} 

 \noindent Next, we show how to apply   the tensor tubal-global Golub-Kahan process to get approximate solutions to the tensor linear system   \eqref{syslintens01}. 
	\begin{proposition}
	Starting from the zero tensor initial guess $\mathscr{  {X}}_{0}$,  	the approximation  $\mathscr{  {X}}_{k}=  \mathbb{V}_{k}\star(\mathscr{  {Y}_{k}}\circledast \mathscr{I}_{ssn_3} )$  of the tensor linear system \eqref{syslintens01} where  $ \mathscr{  {Y}}\in \mathbb{R}^{k\times 1\times n_{3}}$  is such that 
		\begin{equation}
		||\mathscr{B}-\mathscr{A}\star \mathscr{X}_k||_F=  ||\mathscr{E}_1^{(k+1)}\star {\rm \bf a}_{1 }  - \mathscr{ \tilde{C}}_k\star\mathscr{Y}   ||_{T_{\ell_{2}}}
		\end{equation}
	\end{proposition}
	\begin{proof}
		Using  the relation \eqref{equa201},  Proposition \ref{propoortho} and the fact that  $   \mathscr{R}_0=\mathscr{B}=\mathbb{U}_{k+1}\star((\mathscr{E}_1^{(k+1)}\star {\rm \bf a}_{1 }  )\circledast \mathscr{I}_{ssn_3} ) $,  we  get
		\begin{align*}
		||\mathscr{R}_{0}-\mathscr{A}\star \mathscr{X}_k||_F&=||\mathbb{U}_{k+1}\star((\mathscr{E}_1^{(k+1)}\star {\rm \bf a}_{1 }  )\circledast \mathscr{I}_{ssn_3} ) - \mathbb{U}_{k+1} \star(\mathscr{ \tilde{C}}_m\circledast \mathscr{I}_{ssn_3}) \star(\mathscr{  {Y}_{k}}\circledast \mathscr{I}_{ssn_3} ))  ||_{F}\\
		&=||\mathbb{U}_{k+1}\star((\mathscr{E}_1^{(k+1)}\star {\rm \bf a}_{1 }  )\circledast \mathscr{I}_{ssn_3} ) -( \mathscr{ \tilde{C}}_k\star\mathscr{Y}))  \circledast \mathscr{I}_{ssn_3} ||_F\\
		&=||\mathscr{E}_1^{(k+1)}\star {\rm \bf a}_{1 }  -( \mathscr{ \tilde{C}}_k\star\mathscr{Y}))    ||_{T_{\ell_{2}}},
		\end{align*}
		which ends the proof.
	\end{proof}
\noindent 	The approximate solution $\mathscr{X}_k$ produced by this process is given by
	$$\mathscr{  {X}}_{k}=  \mathbb{V}_{k}\star(\mathscr{  {Y}_{k}}\circledast \mathscr{I}_{ssn_3} ),$$
	where ${\mathscr Y}_k$ solves the low-order minimization problem
	$$ \min_{ \mathscr{  {Y}}\in \mathbb{R}^{k\times 1\times n_{3}}}|| ( \mathscr{E}_1^{(k+1)}\star {\rm \bf a}_{1 }  - \mathscr{ \tilde{C}}_k\star\mathscr{Y})    ||_{T_{\ell_{2}}}.$$
	The following algorithm summarizes the main steps to solve the least squares  tensor problem \eqref{syslintens01} using  the tensor tubal-global Golub Kahan. 
	\begin{algorithm}[H]
	 	\begin{algorithmic}[1]
		\REQUIRE $\mathscr{A}\in \mathbb{R}^{n_1\times n_2 \times n_3}$,  $\mathscr{B},\,\in \mathbb{R}^{n_1\times s \times n_3}$, ${\text{kmax}} $  the maximum number of iterations,     and a tolerance $tol$.
			 \ENSURE $\mathscr{X}_k\in \mathbb{R}^{n_2\times s \times n_3}$ the approximate    solution  of the problem  (\ref{syslintens01}). 
			\FOR  {$k=1,\ldots,{\text{kmax}}$}
				\STATE  Apply Algorithm \ref{TTGA} to compute  $\mathbb{V}_{k}$ and $\mathscr{ \widetilde {H}}_k$.
			\STATE Compute $\mathscr{Y}_k=  \displaystyle \text{arg } \min_{ \mathscr{  {Y}}\in \mathbb{R}^{k\times 1\times n_{3}}}|| (\mathscr{E}_1^{(k+1)}\star {\rm \bf a}_{1 }  - \mathscr{ \tilde{C}}_k\star\mathscr{Y})    ||_{T_{\ell_{2}}}	$	
				and  the approximate solution  $\mathscr{X}_{k}= \mathbb{V}_{k}\star(\mathscr{  {Y}}_{k}\circledast \mathscr{I}_{ssn_3} ).
				$	
		\STATE	$~~~~$ \textbf{if}$||\mathscr{R}_{k}||_F<tol$
		
	  $~~~~~~~~$ 	\textbf{Stop}.
		\STATE	$~~~~$ \textbf{end if}
			\ENDFOR	
		\end{algorithmic}
		\caption{ Tensor Tubal-Global  Golub Kahan (TTGK)}
		\label{TTGkh}
\end{algorithm}
	
\section{Numerical experiments }
\noindent In this section we present some numerical tests for  the  tensor tubal-global GMRES     method  when solving  linear tensor problems  \eqref{syslintens0} and we give some comparisons with the tensor global GMRES method define in \cite{elguide1}.  We used only benchmark examples to test the proposed algorithms. All computations were carried out using the MATLAB R2018b environment with  an  Intel(R)  Core i7-8550U CPU $@ 1.80$ GHz  and processor  8 GB. 
\noindent  The   stopping criterion was   $$\frac{||\mathscr{R}_k||_{T_{\ell_{2}}}}{||\mathscr{R}_0||_{T_{\ell_{2}}}} <\epsilon,  $$ where $\epsilon=10^{-6}$ is a chosen tolerance and $\mathscr{R}_k$ the m-th residual associated to    the approximate solution  $\mathscr{X}_k$.  In all the presented tables, we reported the obtained residual norms to achieve the desired convergence, the iteration number and the corresponding cpu-time. 
 \noindent We compared the required cpu-time (in seconds) to achieve the convergence  for   Algorithm \ref{TTGGMRES}     
and T-GGMRES(m),   the  restarted tensor global GMRES introduced in \cite{elguide1}.
\subsection{Example1}
\noindent  The tensor $\mathcal{  {A}}$ is  constructed using $n_3$ frontal  slices. In  this  example,  the frontal  slices are  of  size $n\times n$ and  given as  follows $$A_i={\rm eye}(n)+\frac{i}{2\sqrt{n}} {\rm rand}(n)\;\;\;i=1,\ldots,n_3$$
\noindent The 
right-hand side tensor $\mathcal {B}$ is  constructed   such that the exact
solution $\mathcal{  {X}}^*$ of the tensor linear  equation (\ref{syslintens0}) is  given  by $\mathcal{  {X}}^*={\rm ones}(n,s,n_3).$.  The integer $m$  denotes  the restarted parameter for  the   restarted tensor tubal-global GMRES and also for the restarted  tensor global GMRES.
\begin{table}[H]
	\caption{Results for Example 1. $\epsilon=10^{-6}$, $s=5$ and  $n_3=4$} {
		\begin{tabular}{lcccc}
			\hline  Method & n & $\#$ its. & $\frac{||\mathscr{R}_k||_{T_{\ell_{2}}}}{||\mathscr{R}_0||_{T_{\ell_{2}}}}$    & cpu-time in seconds \\
			\hline   Algorithm \ref{TTGGMRES} &{$500$}&  3 & $1.06 \times 10^{-6}$  & \textbf{0.55}\\
			T-GGMRES&{$500$}&  5 & $1.52 \times 10^{-6}$  & 0.75 \\
	 	\hline  Algorithm \ref{TTGGMRES} &{$1000$}&  3 & $1.26 \times 10^{-6}$  & \textbf{1.53}\\ 
			T-GGMRES  &{$1000$}&  6 & $2.02 \times 10^{-6}$  & 1.92\\
			\hline  Algorithm \ref{TTGGMRES} &{$1500$}&  3 & $3.91 \times 10^{-6}$  & \textbf{3.60}\\
			T-GGMRES(m)   &{$1500$}&  6 & $2.02 \times 10^{-6}$  & 4.02\\
			\hline  
		\end{tabular}
	}
	\label{tab1}
\end{table}
\noindent In Table \ref{tab1}, we reported the obtained relative residual norms, the total number of required iterations to achieve the convergence and the corresponding cpu-times for  the tensor tubal-global GMRES(m) as compared to the   tensor  global  GMRES(m) for different values of the restarted parameter $m$. We gave different values of $n$ and used $m=10$ to fix the size of the projected subspace and to restart the two algorithms. As shown in this table, the    tensor tubal-global GMRES return better results and  	this is due to the fact that it needs lower iterations to achieve the convergence.

\subsection{Example 2}

\noindent  In  this example,  the tensor $\mathscr{A} $ of  size $m_0^2\times m_0^2\times n_3$,   is  the Laplacian tensor constructed by using      the 7-point  discretization of the  three-dimensional Poisson equation (\ref{poiss0})  given  by
\begin{equation}\label{poiss0}
\begin{cases}
-\nabla^2\nu=f, & \varOmega= \left\lbrace  (x,y,z)=0<x,y,z<1\right\rbrace ,\\
\; \;\;\nu=0 &  \text{on }\partial\varOmega, 
\end{cases}
\end{equation}
with  $$\nabla^2\nu=\frac{\partial^2\nu}{\partial x^2}+\frac{\partial^2\nu}{\partial y^2}+\frac{\partial^2\nu}{\partial z^2}. $$  
The right hand  tensor $\mathscr{B} $ was  constructed such  that $\mathscr{B}=\mathcal{  {A}}\star\mathcal{  {X}}^*$  where $\mathcal{  {X}}^*={\rm ones}(m_0^2,s,n_3)$. 
\noindent 	The   mesh step   size is  given by : $\Delta x=\Delta y=\Delta z=h=\frac{1}{(m_0^2+1)} $ where  $\Delta x$, $\Delta y$,
$\Delta z$  are the  step  sizes in  the  x-direction, y-direction and  z-direction, respectively. 
\noindent	Applying the difference formula obtain by the standard central difference approximation gives 
\begin{equation}\label{poiss01}6\nu_{ijk}-\nu_{i-1jk}-\nu_{i+1jk}-\nu_{ij-1k}-\nu_{ij+1k}-\nu_{ijk-1}-\nu_{ijk+1}=h^3 f_{ijk}. \end{equation}  
\noindent	The Laplacian tensor $\mathscr{A} $ can  be 
obtained from the   central difference approximations (\ref{poiss01})  in several forms.   Here,  $\mathscr{A} $ can  be expressed as  a  third order tensor  as  follows 
$${\rm unfold}(\mathscr {A} ) = \left [A_{1}\; A_{2}\; \ldots\; A_{n_{3}} \right ]^T,$$
where 
$$  A_{i-1}=A_{i+1}=  \displaystyle \frac{-1}{h^{3}} \, \begin{pmatrix} 
0 &0 &\ldots&0 &0 \\
0& 1 &0&0 &0 \\
\vdots& \ddots &\ddots&\ddots&\vdots\\
0& \ldots &&1 &0  \\
0& \ldots &0&0 &0
\end{pmatrix}\; {\rm and}\; 
A_{i}=  \displaystyle \frac{-1}{h^{3}}\begin{pmatrix} 
0 &-1 &\ldots&0 &0 \\
-1& 6 &-1&0&0 \\
\vdots& \ddots &\ddots&\ddots&\vdots\\
0& \ldots &-1&6  &-1 \\
0& \ldots &0&-1  &0
\end{pmatrix}. 
$$
\textcolor{red}{ \begin{table}[H]
	\caption{Results for Example 2.  $n_3=m_0^2$, $s=3$ and   $\epsilon=10^{-6}$ } {
		\begin{tabular}{lccccc}
			\hline method &$m$& Size $(m_0^2\times m_0^2\times m_0^2)$ &  $\#$ Iter.   &  $\frac{||\mathscr{R}_k||_{T_{\ell_{2}}}}{||\mathscr{R}_0||_{T_{\ell_{2}}}}$   &cpu-time (in seconds) \\
			\hline 
			Algorithm    \ref {TTGGMRES} &$10$& {$100\times 100\times 100$}&    2& $9.73 \times 10^{-6}$ & \textbf{ 0.52 }    \\ 
			T-GGMRES&$10$& {$100\times 100\times 100$}&    8& $3.85 \times 10^{-6}$ &  0.74    \\ 
			\hline
			Algorithm    \ref {TTGGMRES} &$10$& {$144\times 144\times 144$}&    3& $9.73 \times 10^{-6}$ &  \textbf{1.09 }    \\ 
			T-GGMRES&$10$& {$144\times 144\times 144$}&    7& $1.86 \times 10^{-6}$ &  1.20  \\
			\hline
			Algorithm \ref {TTGGMRES}&$15$&{$225\times 225\times 225$}&    2&  $ 1.51\times 10^{-6}$  &\textbf{2.55}  \\ 	
			T-GGMRES&$15$& {$225\times 225\times 225$}&    6& $3.85 \times 10^{-6}$ &  5.50    \\ 	 
			\hline
		\end{tabular}
	}
	\label{tab2}
\end{table}}

\medskip \noindent Table \ref{tab2} reports on the obtained relative residual norms and the corresponding cpu-time to obtain the desired convergence. As can be seen from this table, the tensor tubal-global GMRES method gives good results  as compared to the  tensor global GMRES. 
We didn't report here the results  for Goulb-Kahan because the algorithm  generally performs well  when solving least-squares equations  with non square problems.

\section{Conlusion} In this paper, we presented some new Krylov subspace methods using the T-product and some new other tensor products. We gave new algebraic properties of these products that allowed us to build new tensor Krylov-subspace based algorithms for solving  tensor equations. We focussed on the tubal-global GMRES and the tubal-global Golub-Kahan algorithms. Some numerical tests on simple examples are also reported.

\section*{Acknowledgements}
The authors would like to thank the editor and anonymous referees for their valuable 
suggestions and constructive comments which improved the quality of the paper.%
 	

\end{document}